\documentclass[a4paper,12pt]{article}
\usepackage[latin1]{inputenc}
\usepackage{amsfonts}
\usepackage{amssymb,amsmath,amsthm}
\usepackage{enumerate}
\usepackage{latexsym}
\usepackage{fancyhdr}
\usepackage[mathlines,pagewise,displaymath]{lineno}
\usepackage[dvipdfmx]{graphicx}
\DeclareSymbolFont{lettersA}{U}{txmia}{m}{it}
 \DeclareMathSymbol{\Indi}{\mathord}{lettersA}{'211}

\makeatletter
    
    \@addtoreset{equation}{section}
  \makeatother

 \def\XXint#1#2#3{{\setbox0=\hbox{$#1{#2#3}{\int}$}
 \vcenter{\hbox{$#2#3$}}\kern-.5\wd0}}

\def\Swiech{{\accent"13S}wie{\hbox{\kern -0.21em\lower 0.79ex\hbox{$\textfont1=\scriptfont1\lhook$}}}ch}
\def\SWIECH{{\accent"13S}WIE{\hbox{\kern -0.26em\lower 0.77ex\hbox{$\textfont1=\scriptfont1\lhook$}}}CH}

\def\P{{\cal P}}

\def\C{{\cal C}}
\def\B{{\cal B}}
\def\to{\rightarrow}
\def\R{{\mathbb R}}
\def\S{{\mathbb S}}
\def\N{{\mathbb N}}
\def\l{\lambda}
\def\L{\Lambda}
\def\G{\Gamma}
\def\ti{\times}
\def\le{\left}
\def\ri{\right}
\def\O{\Omega}
\def\o{\omega}
\def\g{\gamma}
\def\fr{\frac}
\def\p{\partial}
\def\e{\varepsilon}
\def\d{\delta}
\def\s{\sigma}

\def\b{\beta}
\def\a{\alpha}
\def\hat{\widehat}
\def\ol{\overline}
\def\tilde{\widetilde}

\def\under{\underset}
\def\low{\underset}

\def\loc{\mathrm{loc}}
\def\dist{\mathrm{dist}}
\def\supp{\mathrm{supp}}

\def\dis{\displaystyle}

\newtheorem{thm}{Theorem}[section]
\newtheorem{dfn}[thm]{Definition}
\newtheorem{lem}[thm]{Lemma}
\newtheorem{prop}[thm]{Proposition}
\newtheorem{cor}[thm]{Corollary}

\newtheorem{rem}[thm]{Remark}

\addcontentsline{toc}{section}{Acknowledgement}

\title{Weak Harnack inequality for fully nonlinear uniformly parabolic equations with unbounded ingredients and applications}
\author{
\begin{tabular}{c}
\begin{tabular}{ccc}
Shigeaki Koike\footnote{%Supported in part by Grant-in-Aid for Scientific Research (No. 
e-mail: koike@m.tohoku.ac.jp}&$\quad$&Andrzej {\accent"13S}wie{\hbox{\kern -0.21em\lower 0.79ex\hbox{$\textfont1=\scriptfont1\lhook$}}}ch\footnote{e-mail: swiech@math.gatech.edu}\\
Mathematical Institute&&School of Mathematics\\
Tohoku University&&Georgia Institute of Technology\\
Aoba, Sendai 980-8578&&Atlanta, GA 30332\\
JAPAN&&USA
\end{tabular}\\
and\\
Shota Tateyama
\footnote{%Supported by Grant-in-Aid for JSPS Research Fellow 16J02399, 
e-mail: shota.tateyama.p3@dc.tohoku.ac.jp}\\
Mathematical Institute\\
Tohoku University\\
Aoba, Sendai 980-8578\\
JAPAN
\end{tabular}
}
\date{}

\begin{document}

%\thispagestyle{empty}
%\setcounter{page}{0}
%\pagenumbering{roman}
\maketitle
%\tableofcontents
%\newpage
\pagenumbering{arabic}

\begin{abstract}
The weak Harnack inequality for $L^p$-viscosity supersolutions of 
fully nonlinear second-order uniformly parabolic partial 
differential equations 
with unbounded coefficients and inhomogeneous terms is proved. 
It is shown that H\"older continuity of $L^p$-viscosity solutions is derived from the weak Harnack inequality for $L^p$-viscosity supersolutions. 
The local maximum principle for $L^p$-viscosity 
subsolutions and the Harnack inequality for $L^p$-viscosity solutions are also obtained.
Several further remarks are presented when equations have superlinear growth 
in the first space derivatives. 

{\it 
\begin{center}
\begin{tabular}{ll}Keywords:&fully nonlinear parabolic equations, viscosity 
solutions, Harnack\\
&inequality, maximum principle.\\
2010 MSC:&49L25, 35D40, 35B65, 35K55, 35K20, 35K10.\\
&\\
\end{tabular}
\end{center}
}
\end{abstract}

\section{Introduction}%%%%%%%%%%%%%%%%%%%%%%SECTION#1
  \label{sec:intro}

The seminal paper \cite{Caf} of L.A. Caffarelli was the most influential in the development of modern regularity theory for viscosity solutions of fully nonlinear uniformly elliptic partial differential equations (PDE for short). Various results were proved there, including Harnack inequality, 
$C^\alpha, C^{1,\alpha}, C^{2,\alpha}$ and $W^{2,p}$ estimates, 
and the reader can find a more detailed and complete account of them in \cite{CafCab}. Around the same time similar results
like Harnack inequality, $C^\alpha$ and
$C^{1,\alpha}$ estimates for viscosity solutions were also proved by different methods in \cite{T2, T1, T3}.
%However, it was necessary to assume that inhomogeneous terms are continuous since the standard definition of viscosity solutions is not available when PDE have measurable ingredients. 
In order to treat PDE with measurable terms, the notion of $L^p$-viscosity solution of 
fully nonlinear uniformly elliptic PDE was introduced in \cite{CCKS} and a similar idea was also considered in \cite{W}. 
L. Wang in \cite{W, W1} extended regularity results of \cite{Caf} to viscosity solutions of fully nonlinear uniformly parabolic PDE. 
Later, $L^p$-viscosity solutions of parabolic PDE were studied in \cite{CFKS, CKS}. 

The main ingredient in the theory is the Aleksandrov-Bakelman-Pucci (ABP for short) maximum principle, which gives the 
$L^\infty$-estimates in terms of the $L^p$-norms of the inhomogeneous terms. The ABP maximum principle for viscosity solutions
of fully nonlinear uniformly parabolic PDE was proved in \cite{W}.
In \cite{CCKS}, the ABP maximum principle was proved for $L^p$-viscosity solutions of uniformly elliptic PDE which are uniformly Lipschitz continuous in the first derivatives. It was later extended for elliptic and parabolic PDE to equations which are not uniformly Lipschitz continuous in the first derivative terms in \cite{KS2}, where the Lipschitz 
coefficient functions (as functions of $x$ and $t$) belong to some $L^q$ 
spaces. 
The second ingredient of the regularity theory of \cite{Caf} is 
 the Harnack inequality for viscosity solutions as well as the weak Harnack inequality and the local maximum principle.
 Such results for non-divergence form equations started with the work of Krylov and Safonov \cite{KrS} and the results for
 strong solutions can be found in classical books \cite{GilTru83, L}. Results for viscosity solutions first appeared in
 \cite{Caf, T1} (see \cite{CafCab}). General form of the weak Harnack inequality for $L^p$-viscosity supersolutions of 
fully nonlinear elliptic PDE (which implies the H\"{o}lder continuity of 
$L^p$-viscosity solutions) was proved in \cite{KS3}, using the ABP estimates of \cite{KS2}, while a general 
local maximum principle for $L^p$-viscosity solutions can be found in \cite{KS5}. The corresponding results for 
viscosity solutions of uniformly parabolic PDE were proved in \cite{W}, however only for equations which are
uniformly Lipschitz continuous in the first derivatives.
In this paper we want to extend them to $L^p$-viscosity solutions of more general equations. The relevant equations are the
parabolic
extremal equations 
\[
u_t +\P^\pm(D^2u) \pm\mu |Du|+f=0 \quad\mbox{in } Q, 
\]
where $f \in L^p(Q)$ and $\mu\in L^q(Q)$. 

%In a classical result by Krylov and Safonov \cite{KrS} via a probabilistic  argument, the Harnack inequality for linear PDE of %non-divergent type was proved while  proofs utilized only PDE tools for strong solutions can be found in %the textbook by %Lieberman 
%\cite{GilTru83, L} for instance. 
%It is worth mentioning a paper by Rey \cite{Rey} for another approach but only to strong solutions. 

%For viscosity solutions of uniformly parabolic PDE, 
%Wang \cite{W} obtained regularity results 
%corresponding to those for elliptic PDE by Caffarelli \cite{Caf}. 
%All of these results deal with fully nonlinear uniformly parabolic PDE only with bounded coefficients to the first derivatives. 
%If we use the barrier functions of \cite{W} and  \cite{ImbSil13} 
%for parabolic PDE with unbounded coefficients, 
%then we have to suppose that the coefficients are small in some norms. 

In this manuscript, combining the argument from \cite{ImbSil13} with 
the ABP maximum principle of \cite{KS2}, we  first show 
the weak Harnack inequality when the $L^q$-norm of the coefficient function $\mu$ is small. 
We then avoid this smallness assumption by the introduction
of a new ``heat kernel'' like barrier 
function in our proof of the weak Harnack inequality. 
We will use global estimates on strong solutions of fully nonlinear parabolic equations from a recent paper by Dong, Krylov and Li \cite{DKX}.
We  remark that the weak Harnack inequality yields the (local) 
H\"older estimate. 
In order to establish the Harnack inequality, following the argument 
of \cite{CafCab} (see also \cite{KS5}), 
we also obtain the corresponding local maximum principle. 
We refer to \cite{W} and \cite{ImbSil13} for the other approach. 
We also present some results when the PDE contains first space derivative terms which may grow superlinearly.

This paper is organized as follows. In Section 2, we recall the definition of $L^p$-viscosity solution for parabolic PDE, 
its properties and known results. 
Section 3 is devoted to a proof of the weak Harnack inequality 
for $L^p$-viscosity supersolutions. 
In Section 4 we first
establish the local H\"older continuity estimate using the weak Harnack inequality.  
For the completeness of the theory, 
we show the local maximum principle for $L^p$-viscosity subsolutions 
by a parabolic version of the argument of \cite{CafCab} and then obtain the Harnack inequality. 
In Section 5, we present some results for PDE which may contain superlinearly growing gradient terms.

%\pagenumbering{arabic}

%\begin{center}Acknowledgement\end{center}
%The authors would like to thank Professor G. Akagi for letting us know a result on compact embeddings in \cite{S}.
%\end{acknowledgement}

\section{Preliminaries}%%%%%%%%%%%%%%%%%%%%%SECTION#2

%\subsection{Definitions}
We fix $n\in\N$, a bounded domain $\O\subset\R^n$, and $T>0$. 
%We use $\langle\cdot,\cdot\rangle$ for the inner product in $\R^n$.  
We denote by $\S^n$ the set of all $n\ti n$ symmetric matrices 
with the standard order.

Given $F\colon\O\ti(0,T]\ti\R^n\ti\S^n\to\R$, we are concerned with the following fully nonlinear parabolic PDE: 
\begin{equation}
   \label{eq:dfn}
    u_t+F(x, t, Du, D^2u)=0 \quad\mbox{in }\O\ti(0,T], 
\end{equation}
where $Du$ and $D^2u$, respectively, denote 
the first and second derivatives with respect to $x\in \R^n$, 
$u_t$ is the time derivative, and $F$ is at least measurable 
with respect to all the variables. 
We will write $u_{x_k}, u_{x_kx_\ell}$ for $\fr{\p u}{\p x_k}, \fr{\p^2 u}{\p x_\ell\p x_k}$, respectively. 

In what follows, we assume that $F$ is uniformly parabolic, i.e. that
there exist $0<\l\leq\L<\infty$ such that 
\begin{equation}\label{UP}
\P^-_{\l, \L}(X-Y)\leq F(x,t,\xi,X)- F(x,t,\xi,Y)\leq \P^+_{\l, \L}(X-Y)
\end{equation}
for all $(x, t, \xi, X, Y)\in\O\ti(0,T]\ti\R^n\ti\S^n\ti\S^n$, where $\P^\pm_{\l, \L}\colon\S^n\to\R$ are defined by
\[\P^+_{\l, \L}(X):=\max\{ -\mathrm{Tr}(AX) \ | \ A\in\S^n, \l I\leq A\leq \L I\},\]
\[\P^-_{\l, \L}(X):=\min\{ -\mathrm{Tr}(AX) \ | \ A\in\S^n, \l I\leq A\leq \L I\}\]
for $X\in\S^n$, where $I$ denotes the $n\ti n$ identity matrix. 
Since  we fix $0<\l\leq \L$ in this paper, 
we simply write $\P^\pm$ for $\P^\pm_{\l. \L}$. 
For properties of $\P^\pm$, we refer for instance to \cite{CCKS}. 
 
Setting $Q:=\O \ti (0,T]$, we denote the parabolic boundary of $Q$ by
\[
\p_p Q:=\O\ti \{0\}\bigcup\p\O\ti [0,T).\]
The parabolic distance is defined by
\[
d((x,t),(y,s)):=\sqrt{|x-y|^2+|t-s|}.
\]
For $U$, $V\subset\R^{n+1}$, we define the distance between $U$ and $V$
 \[\dist(U,V):=\inf\le\{d((x,t),(y,s)) \,|\, (x, t)\in U, (y, s)\in V \ri\}.\]
We will write ${\rm diam}(Q)$ for the diameter of $Q$ (measured with respect to the parabolic distance) and ${\rm diam}(\O)$
for the diameter of $\O$.

We will use the anisotropic Sobolev spaces. For $1\leq p\leq\infty$, 
\[
W^{2,1}_p(Q):=\le\{ f\in L^p(Q) \ \le| \ f_{x_k},f_{x_k x_\ell}, f_t\in L^p(Q) \ri. \ (1\leq k,\ell\leq n)\ri\}, 
\]
and
\[W^{2,1}_{p,\loc}(Q):=\le\{ f\in W^{2,1}_p(Q') \ | \ \forall Q'\Subset Q \ri\}.\]
Here and later, $Q'\Subset Q$ means dist$(Q',\p_pQ)>0$. 
We define the norm for  $f\in W^{2,1}_p(Q)$ by 
\[\|f\|_{W^{2,1}_p(Q)}:=\|f\|_{L^p(Q)}+\|f_t\|_{L^p(Q)}+\sum_{k=1}^n\|
         f_{x_k}\|_{L^p(Q)}+\sum_{k,\ell=1}^n\|
         f_{ x_k x_\ell}\|_{L^p(Q)}.\]
We will also use the anisotropic Sobolev spaces 
\[W^{1,0}_p(Q):=
\left\{ 
f\in L^p(Q) \ \le| \ f_{x_k}\in L^p(Q) \ri. \ (1\leq k\leq n)\ri\}\]
for $1\leq p\leq \infty$, equipped with the norm
\[\|f\|_{W^{1,0}_p(Q)}:=\|f\|_{L^p(Q)}+\sum_{k=1}^n\|
         f_{x_k}\|_{L^p(Q)}.\]

     We denote by $C^{2,1}(Q)$ the space of functions $%\phi
     u\in C(Q)$ such that 
$u_t,u_{x_k},u_{x_kx_\ell}\in C(Q)$ for $1\leq k,\ell\leq n$. 
For $0<\alpha\leq 1$ we denote by $C^\alpha(Q)$ the space of functions which are $\alpha$-H\"older continuous in $Q$
with respect to the parabolic distance. We denote by $W^{k}_p(\Omega)$, $k=1,2,...$, the standard Sobolev spaces.

We recall the notion of $L^p$-viscosity solutions of parabolic PDE (\ref{eq:dfn}). 
To this end, we denote by $B_r(x)$ the open ball in $\R^{n}$ with the radius $r>0$ and the center $x$, 
and define the parabolic cylinders
\[Q_r(x,t):=(x,t)+(-r,r)^n\ti(-r^2, 0].\]   %%CHANGE BALL BY CUBE

%%%%%%%%%%%% Def of VS

\begin{dfn}
Let $Q'$ be a relatively open subset of $Q$. A function $u\in C(Q')$ is said to be an $L^p$-viscosity subsolution (resp., 
supersolution) 
of $(\ref{eq:dfn})$ if for $\phi\in W^{2,1}_{p,\loc}(Q')$, 
we have
\[
\lim_{\e\to 0}ess\under{(y,s)\in Q_\e(x,t)}{\, \inf}
\le\{\phi_t(y,s)+F(y,s,D\phi (y,s),D^2\phi (y,s))
\ri\}\leq 0
\]
\[\left(\mbox{resp., }\lim_{\e\to 0}ess\under{(y,s)\in Q_\e(x,t)}{\, \sup}
\le\{\phi_t(y,s)+F(y,s,D\phi (y,s),D^2\phi (y,s))
\ri\}\geq 0\right)
\]
provided that $u-\phi$ attains a maximum (resp., minimum) 
at $(x,t)\in Q'$ over some parabolic cylinder $Q_r(x, t)\subset Q'$. 
A function $u\in C(Q')$ is said to be an $L^p$-viscosity solution of $(\ref{eq:dfn})$  
  if $u$ is an $L^p$-viscosity subsolution and supersolution of $(\ref{eq:dfn})$.
%\end{enumerate}
\end{dfn}

%%%%%%%%%%%%%  Remark 2.2

\begin{rem}
We note that $W^{2,1}_{p,\loc}(Q)\subset C(Q)$ for $p>\fr{n+2}{2}$ and if $\O$ is regular enough (e.g. if $\partial\O$ is
$C^{1,1}$) then $W^{2,1}_{p}(Q)\subset C^\alpha(Q)$ for $\alpha=2-(n+2)/p$ is a bounded imbedding for $\fr{n+2}{2}<p<n+2$.
If $u\in W^{2,1}_{p,\loc}(Q)$ for $p>n+2$ then $u_{x_i}\in C^\alpha$ for $\alpha=1-(n+2)/p$ (see e.g. \cite{LSU}).
Also, it is known that if  $p>\fr{n+2}{2}$ and $u\in W^{2,1}_p(Q)$, then $u_t$,  $u_{x_i}$ and $u_{x_ix_j}$ ($1\leq i,j\leq n$) exist $a.e.$ in $Q$ (see {\rm \cite{CFKS}}). 
\end{rem}

%%%%%%%%%%%%%%%%%%     SubSection   ABP maximum principle
%%%%%%%%%%%%%%%%%%%

In this section, we recall the ABP maximum principle for $L^p$-viscosity subsolutions of the following  extremal uniformly parabolic equations:
\begin{equation}\label{P-}
u_t+\P^-(D^2u)-\mu |Du|-f=0 \quad \mbox{in } Q , 
\end{equation}
where 
\[f\in L^p(Q), \quad\mbox{and}\quad \mu\in L^q(Q).\]
We will suppose that the powers $p$ and $q$ satisfy the condition
\begin{equation}\label{Apq1}
q>n+2,\,\,\,p_1<p\leq q,
\end{equation}
%one of 
%the conditions: 
%\begin{equation}\label{Apq1}
% n+2<p\leq q \quad \mbox{or} \quad p=n+2<q, 
%\end{equation}
%\begin{equation}\label{Apq2}
%p_1<p< n+2<q ,
%\end{equation}
where $p_1=p_1(n,\fr{\L}{\l})\in [\frac{n+2}{2},n+1)$ is     
the constant, which gives a range where 
the ABP maximum principle holds, see e.g. \cite{KS2}.  

%%%%%%%%% Prop. 2.3   Proposition (Existence of strong solution)

\begin{prop}\label{prop:p_1}(cf. Theorem 2.8 in \cite{CKS}, Proposition 3.3 in \cite{KS2}) 
For  $p>p_1$, there exists a constant $C=C(n,\L,\l,p)>0$ 
such that for $f\in L^p(Q)$, there exists $u\in C(\ol{Q})\cap W^{2,1}_{p,\loc}(Q)$ such that 
\[
\le\{\begin{array}{rl}
u_t+\P^+(D^2u)-f(x,t)=0&{\it a.e.} \ \mathrm{in} \ Q,\\
u= 0 &{\rm on} \ \p_p Q,
\end{array}\ri.
\]
and
\[-C\| f^-\|_{L^p(Q)}\leq u\leq C\| f^+\|_{L^p(Q)} \quad \mathrm{in} \ Q.\]
Moreover, for $Q'\Subset Q$, there is $C'=C'(n,\L,\l,p,T,{\rm diam}(\O),\dist(Q',\p_pQ))>0$ such that  
\[
\|u\|_{W^{2,1}_p(Q')}\leq C'\|f\|_{L^p(Q)}.\]
\end{prop}

We emphasize that the dependence of constants on various parameters sometimes may mean that a constant may blow up
as a parameter converges to $0$, for instance the constant $C'$ in Proposition \ref{prop:p_1} may blow up as $\lambda\to 0$. The precise dependence of constants on $T, {\rm diam}(\O), {\rm diam}(Q)$ can often be found by scaling.

We state in Proposition \ref{thm:ABP} a scaled version of the  ABP maximum principle for $L^p$-strong and $L^p$-viscosity solutions of 
\eqref{P-} based on the results of \cite{KS2}. 
For $u\in C(\ol{Q})$, we introduce the set 
\[Q_+[u]:=\le\{(x,t)\in Q \ 
\le| \ u(x, t)>\low{\p_pQ}\sup\, u^+ \ri.\ri\}. \]
We denote by $L^p_+(Q)$ for the set of all nonnegative functions in $L^p(Q)$.

%%%%%%%%  Remark 2. 4

\begin{rem}
The non-scaled statement of the classical ABP maximum principle for strong solutions in Proposition 3.2 of \cite{KS2} was slightly incorrect and might be confusing. The exact ABP inequality from \cite{Tso} is
\begin{equation}\label{MP1}
\|u\|_{L^\infty(Q)}
\leq\|\psi\|_{L^\infty(\p_p Q)}+C_1d_{\O}^{\frac{n}{n+1}}\exp\le(C_2d_{\O}^{-1}\|\mu\|^{n+1}_{L^{n+1}(Q)}\ri)\|f\|_{L^{n+1}(Q)},
\end{equation}
where $d_{\O}={\rm diam}(\O)$, which behaves differently from the one in \cite{KS2} when ${\rm diam}(\O)$ is small. This however does not affect the results of \cite{KS2} since the proofs
there only used (\ref{MP1}) in parabolic cylinders of fixed size which contained $Q$ and did not depend on ${\rm diam}(\O)$.
\end{rem}

%%%%%%%%  Proposition  2. 5

\begin{prop}\label{thm:ABP}(see Proposition 3.6, Theorem 3.10 of \cite{KS2}) 
Let \eqref{Apq1} hold and let
 $f\in L^p_+(Q),\mu\in L^q_+(Q)$. 
There exists a constant $C_1=C_1(n,\L,\l,p,q,d_Q^{1-(n+2)/q}\|\mu\|_{L^{q}(Q)})>0$ 
such that 
if $u\in C(\ol{Q})$ is either an $L^p$-strong or an $L^p$-viscosity subsolution of  $(\ref{P-})$, then 
\begin{equation}\label{ABP2}
\under{Q}{\sup}\,u\leq 
\under{ \p_p Q }{ \sup }\, u +C_1d_Q^{2-\frac{n+2}{p}}\|f\|_{L^{p}(Q)},
\end{equation}
where $d_Q={\rm diam}(Q)$.
Moreover, if $v\in C(\ol{Q})$ is an $L^p$-viscosity subsolution of $(\ref{P-})$ in $Q_+[v]$ then
\begin{equation}\label{ABP3}
\under{Q}{\sup}\,v\leq 
\under{ \p_p Q }{ \sup }\, v^+ +C_1d_Q^{2-\frac{n+2}{p}}\|f\|_{L^{p}(Q_+[v])}.
\end{equation}
\end{prop}
We remark that when $p\geq n+2$ then (\ref{ABP2}) can be made more precise based on (\ref{MP1})
or on a scaled version of (\ref{MP1}) in a unit cylinder. 
We also remark that it can be proved that if \eqref{Apq1} holds then an $L^p$-strong subsolution of $u_t+\P^\pm(D^2u)\pm\mu |Du|=f$ is an $L^p$-viscosity subsolution
of %%%%%%%%%%%% 
those. %%%%%%%  ADDED
We refer to \cite{KS3}, Section 3, for such a proof in the elliptic case. Similar statement holds for 
$L^p$-viscosity supersolutions.
\begin{proof}[Proof of Proposition \ref{thm:ABP}]
To see why \eqref{ABP2} is true we notice that the function $w(x,t)=u(d_Qx,d_Q^2t)$ is an $L^p$-strong or an $L^p$-viscosity subsolution of
\[
w_t+\P^-(D^2w)-\tilde\mu |Dw|-\tilde f=0 
\]
in a unit cylinder $Q_1$, where $\tilde\mu(x,t)=d_Q\mu(d_Qx,d_Q^2t)$ and $\tilde f(x,t)=d_Q^2f(d_Qx,d_Q^2t)$.
Then, by the estimates of \cite{KS2}
\[
\under{Q_1}{\sup}\,w\leq 
\under{ \p_p Q_1 }{ \sup }\, w +C_1(n,\L,\l,p,q,\|\tilde\mu\|_{L^{q}(Q_1)})\|\tilde f\|_{L^{p}(Q_1)}.
\]
It remains to notice that $\|\tilde\mu\|_{L^{q}(Q_1)}=d_Q^{1-(n+2)/q}\|\mu\|_{L^{q}(Q)}$ and
$\|\tilde f\|_{L^{p}(Q_1)}=d_Q^{2-\frac{n+2}{p}}\|f\|_{L^{p}(Q)}$. 

Estimate \eqref{ABP3} is proved similarly by
rescaling and adding to $w$ a subsolution of an extremal equation in a bigger cylinder to eliminate $\tilde f$, which can be found using
Proposition 3.5 of \cite{KS2} (or using Proposition \ref{prop:exist}
 if $q\geq p>n+2$). The reader can find a similar argument in the proof
of Proposition 2.8 of \cite{KS2}.
\end{proof}

A result similar to Proposition \ref{prop:exist} can be found in $\cite{KS2}$ (see Proposition 3.5 there). Using global $W^{2,1}_p$ estimates by Dong, Krylov and Li in \cite{DKX}, 
we present a slightly different existence result.

%%%%%%%%%%   Proposition   2. 6 

\begin{prop}\label{prop:exist}%(cf. \cite{DKX})
Assume that $\p\O$ is $C^{1,1}$ and $q\geq p>n+2$. 
Let $\mu\in L^q_+(Q)$, 
$\psi\in W^{2,1}_p(Q)
\cap C(\ol{Q})$
and $f\in L^p(Q)$. The equation
\begin{equation}\label{Exists}
\le\{\begin{array}{rr}
u_t+\P^+(D^2u)+\mu |Du|-f= 0 &{\it a.e.} \ \mathrm{in} \ Q,\\
%v_t+\P^+(D^2v)+\mu(x,t)|Dv|\leq f(x,t) &{\it a.e.} \ \mathrm{in} \ Q, \\
u= \psi &\mathrm{on } \ \p_p Q
\end{array}
\ri.
\end{equation}
has an $L^p$-strong solution $u\in
C(\ol{Q})\cap 
W^{2, 1}_{p}(Q)$. The solution $u$ satisfies
\begin{equation}\label{MP}
\|u\|_{L^\infty(Q)}
\leq\|\psi\|_{L^\infty(\p_p Q)}+C_1d_Q^{2-\frac{n+2}{p}}\|f\|_{L^{p}(Q)}
\end{equation}
(where $C_1$ is the constant from \eqref{ABP2}), and 
\[%\label{LpEst}
\|u\|_{W^{2,1}_p(Q)}
\leq C_2\le(\|\psi\|_{W^{2,1}_p(Q)}+\|f\|_{L^p(Q)}\ri)
\]
for some constant $C_2=C_2(n,\L,\l,p,q,\|\mu\|_{L^q(Q)},T,{\rm diam(\O)},\partial\O)>0$.
\end{prop}

%%%%%%%%%%  Remark   2.  7 

\begin{rem}
The function $u$ in Proposition \ref{prop:exist} is also 
an $L^p$-viscosity %%%%%%%%%%%  CHANGED  
solution of \eqref{Exists}. 
We also note that  $p>n+1$ is assumed in \cite{DKX} while we assume 
$p> n+2$.
\end{rem}

\begin{proof}
Let $f^j$, $\mu^j\in C(\ol Q)$ be such that $\| f^j-f\|_{L^p(Q)}+\| \mu^j-\mu\|_{L^q(Q)}\to 0$, and 
$(f^j,\mu^j)\to (f,\mu)$ almost everywhere in $Q$  as $j\to \infty$.  
Let $u^j\in C(\ol{Q})\cap C^{2,1}(Q)$ be the classical solution of  
\begin{eqnarray*}%\label{Strongsol}
\le\{ \begin{array}{rl}
u^j_{t}+\P^+(D^2u^j)+\mu^j|Du^j|-f^j=0&\mbox{ in } Q,\\ 
u^j=\psi & \mbox{on } \p_p Q.
\end{array}\ri.
\end{eqnarray*}
Here and later, $C>0$ stands for various constants depending only on known quantities. 
We know from \cite{DKX} that %\cite{W} that (IS IT \cite{W} OR \cite{DKX}? WHERE IS \cite{DKX}? USED HERE?)  
\begin{equation}\label{eq1}
\|u^j\|_{W^{2,1}_{p}(Q)}\leq C\le(\le\|\mu^j|Du^j|-f^j\ri\|_{L^{p}(Q)}+\| \psi\|_{W^{2,1}_{p}(Q)}\ri).
\end{equation}
It is also known (e.g. Lemma 3.3 in \cite{LSU}) that for sufficiently small $\e>0$, 
\begin{equation}\label{eq:ASI}
\|D u^j\|_{L^\infty(Q)}\leq \e^{\a_1} \le( \|D^2u^j\|_{L^p(Q)} + 
\| u^j_{t}\|_{L^{p}(Q)}\ri)+\e^{-\a_2}C\| u^j\|_{L^p(Q)},
\end{equation}
where $\a_1=1-\fr{n+2}{p}>0$ and $\a_2=1+\fr{n+2}{p}>0$. 
Combining \eqref{eq:ASI} with \eqref{eq1}, in view of the global estimates in 
\cite{DKX}, we have  
\[%\label{eq.1}
\begin{array}{rl}
\|u^j\|_{W^{2,1}_p(Q)}
\leq&%C\le(\|u_{j,t}+\P^-(D^2u_j)\|_{L^p(Q)}+\|\psi\|_{W^{2,1}_p(Q)}\ri)\\
%=&
C\le( \|f^j\|_{L^p(Q)}+\le\|\mu^j Du^j\ri\|_{L^p(Q)}+\|\psi\|_{W^{2,1}_p(Q)}\ri)\\
\leq &C\|\mu^j\|_{L^p(Q)}\le\{\e^{\a_1}
\le(\|D^2u^j \|_{L^p(Q)}+\|u^j_{t}\|_{L^p(Q)}\ri)
+\e^{-\a_2}\| u^j\|_{L^p(Q)}\ri\}\\
&+C\le(\| f^j\|_{L^p(Q)}+\| \psi\|_{W^{2,1}_p(Q)}\ri).
\end{array}
\]
Hence, for an appropriate $\e>0$ (depending on $\|\mu^j\|_{L^p(Q)}$), using the ABP maximum principle, we obtain
\begin{equation}\label{W21p}
\| u^j\|_{W^{2,1}_p(Q)}\leq C\le( \|f^j\|_{L^p(Q)}+\| \psi\|_{W^{2,1}_{p}(Q)}\ri).
\end{equation}

%We now assume $p=n+2<q$. 
%Instead of $(\ref{eq:ASI})$, we use the following inequality from Lemma 3.3 in \cite{LSU}: for any $r>\fr{n+2}{2}$, there exists a %constant $C=C(r)>0$ such that   
%\begin{equation}\label{eq:0}
%\|D u^j\|_{L^r (Q)}\leq \e^{\a_3}\le(\|u^j_{t}\|_{L^{n+2}(Q)}+\| D^2u^j\|_{L^{n+2}(Q)}\ri)
%+C\e^{-\a_4}\| u^j\|_{L^{n+2}(Q)},
%\end{equation}
%where $\a_3=\fr{n+2}{r}>0$ and $\a_4=2-\fr{n+2}{r}>0$. 
%Setting $r:=\fr{q(n+2)}{q-n-2}>\fr{n+2}{2}$, by the H\"{o}lder inequality, we have 
%$$
%\begin{array}{rl}
%\|\mu^jD u^j\|_{L^{n+2}(Q)}\leq&\|\mu^j\|_{L^q(Q)}\|D u^j\|_{L^{r}(Q)}\\
%\leq& \|\mu^j\|_{L^q(Q)}\le\{\e^{\a_3}\le(\| u^j_{t}\|_{L^{n+2}(Q)}+\| D^2u^j\|_{L^{n+2}(Q)}\ri)+C\e^{-\a_4}\| u^j\|_{L^{n+2}(Q)}\ri\}.
%\end{array}$$
%Using $(\ref{eq:0})$ with this estimate, we obtain \eqref{W21p} when $p=n+2$. 

Since (by anisotropic Sobolev imbeddings) the functions $u^j$ are equicontinuous in $C(\ol{Q})$ and $u^j_{x_i}, i=1,...,n$, are locally
equicontinuous, by taking a subsequence, we can assume that there exists $u\in W^{2,1}_p(Q)\cap C(\ol{Q})$ satisfying
\eqref{MP} and \eqref{W21p}
such that $u^j\rightharpoonup u$ in $W^{2,1}_p(Q)$,
$u^j\to u$ in $C(\ol{Q})$ and $u^j_{x_i}\to u_{x_i}$ locally uniformly. Thus $-\mu^j|Du^j|+f^j\to -\mu|Du|+f$
in $L^p_{\rm loc}(Q)$. It is then standardized by the techniques of \cite{CKS}
to obtain that $u$ is an $L^p$-viscosity solution and hence an $L^p$-strong solution of
\[
u_{t}+\P^+(D^2u)=g,
\]
where $g=-\mu|Du|+f$, which concludes the proof.
\end{proof}

%%%%%%%%%%%%%%%%%%%%%%%%%%%%%
%%%%%%%%%%%%%%%%
%%%%%%%%%%%%%%          Section 3
%%%%%%%%%%%%%%%%%%
%%%%%%%%%%%%%%%%%%%%%%%%%%%%%%

\section{The weak Harnack inequality}%%%%%%%%%%%%%%%%%%%%%%%%%%%SECTION#3
%We fix $R>\sqrt{n}$ and 
In what follows, we set $\O:=(-10,10)^n,T=10$ and
\[
Q:=(-10,10)^n\ti (0,10].\]
Although we need to suppose $\p\O\in C^{1,1}$ to use Proposition \ref{prop:exist}, for the sake of simplicity of the presentation, we will assume that  %%%%%%%  CHANGED  
the boundary of cubes are $C^{1,1}$. Otherwise we would have to use a smooth domain similar to $(-10,10)^n$. 
We refer to \cite{KS3} for such an argument.

In this section, we show the weak Harnack inequality for nonnegative $L^p$-viscosity supersolutions of
\begin{equation}\label{P+}
u_t +\P^+(D^2u) +\mu |Du|+f=0 \quad\mbox{in } Q, 
\end{equation}
where $f \in L^p_+(Q)$ and $\mu\in L^q_+(Q)$.

%%%%%%%%%%%
%%%%%%%%%%%      3. 1.  A restricted case
%%%%%%%%%%%

\subsection{A restricted case}

In order to show the weak Harnack inequality for nonnegative $L^p$-viscosity %%%%%%%  CHANGED 
 supersolutions of $(\ref{P+})$ with $f\in L^p_+(Q)$ 
and $\mu\in L^q_+(Q)$, %%%%%%%%  ADDED  "+"
we follow the standard argument as in \cite{ImbSil13} except for a new barrier function, which will be constructed in Lemma \ref{lem:barrier}. 
However, for this purpose, we first have to show the weak Harnack inequality under 
a restricted setting.

%%%%%%%%%%%%%%%%%%%%%%%%%     Theorem  3. 1. 

\begin{thm}
\label{thm:weak0}
Assume that $(\ref{Apq1})$
holds, $f\in L^p_+(Q)$ and $\mu\in L^q_+(Q)$.   
Then, there exist constants $\e_0=\e_0(n,\L,\l,p,q)>0$, 
$\d_0=\d_0(n,\L,\l, p,q)>0$ and 
$C_0=C_0(n,\L,\l,p,q)>0$ 
such that if 
\begin{equation}\label{wHsmall}
\|\mu\|_{L^{p}%n+1}                           CHANGED 
(Q)}\leq \d_0,
\end{equation}
then 
any nonnegative $L^p$-viscosity supersolution $u$
%\in C(\overline{Q})$ 
of \eqref{P+} satisfies  
\[%\label{WHI}
\le(\int_{J_1} u^{\e_0} \,dxdt\ri)^{\fr{1}{\e_0}}\leq C_0\le(\low{J_2}{\inf}\, u +\|f\|_{L^{ p}(Q)}\ri),
\]
where 
\[J_1:=(-1, 1)^n\ti (0,2^{-1} ],\quad\mbox{and}\quad J_2:=(-1, 1)^n\ti (9, 10].\]
\end{thm}

We remark that the statement of Theorem \ref{thm:weak0} also holds for nonnegative $L^p$-strong supersolutions of
\eqref{P+}.

In order to prove Theorem \ref{thm:weak0}, %and \ref{thm:weak2}, 
we first construct a strong subsolution of an extremal equation. 
To this end, we use the following cubes:
\[
K_1:=(-1, 1)^n\ti(0, 1], \quad\mbox{and}\quad 
K_2:=(-3, 3)^n\ti(1, 10].\]

We recall a barrier function from Lemma 2.4.16 of \cite{ImbSil13} 
(see also \cite{W}). 
We can also construct one by the same manner as in Lemma \ref{lem:barrier} here.

\vspace{1.5cm}

%%%%%%%%%%%%%%%%%%%%
%%%%%%%%%%%%%%%%%%%%     Fig   1
%%%%%%%%%%%%%%%%%%%%

%%%%%%%%%%%%%%%%%%%%
%%%%%%%%%%%%%%%%%%%%     Fig   1
%%%%%%%%%%%%%%%%%%%%

\begin{picture}(450,250)
\put(0,50){\vector(1,0){400}}%\UTF{00C2}\UTF{0089}\UTF{00C2}\UTF{00A1}\UTF{00C2}\UTF{008E}\UTF{00C2}\UTF{00B2}%
\put(200,35){\vector(0,1){250}}%\UTF{00C2}\UTF{008F}c\UTF{00C2}\UTF{008E}\UTF{00C2}\UTF{00B2}%
\put(50,50){\framebox(300,200)}%Q
\put(180,50){\framebox(40,10)}%{\line(0,1){15}}%J1
\put(180,50){\framebox(40,20)}%K1
\put(190,50){\framebox(20,5)}%K_{1/4}
\put(150,70){\framebox(100,180)}%K2
\put(185,45){$\nearrow$}
\put(175,37){$K_{1/4}$}
\put(189,72){$\searrow$}
\put(184,80){$1$}
%\put(240,110){\line(0,1){15}}%J2
%\put(208,110){\line(1,0){32}}%J2
%\put(208,58){\line(0,-1){8}}%J1
%\put(240,58){\line(0,-1){8}}%J1
%\put(208,58){\line(1,0){32}}%J1
\put(180,230){\framebox(40,20)}%J2
%\put(170,65){\framebox(110,60)}%K2
%\put(135,90){\framebox(110,60)}
%\put(180,90){$K_2$}
\put(220,160){$K_2$}%K2
\put(200,220){$\nwarrow$}
\put(215,210){$9$}
\put(200,63){$\swarrow$}
\put(207,76){$\frac12$}
\put(80,90){$Q$}
\put(160,54){$J_1\rightarrow$}
%\put(225,45){$\nwarrow$}
\put(220,57){$\} K_1$}
\put(218,40){$1$}
\put(160,230){$J_2\rightarrow$}
\put(203,37){$0$}
\put(380,42){$x$}
\put(248,255){$3$}
\put(190,270){$t$}
\put(345,40){$10$}
\put(200,253){$\swarrow$}
\put(210,259){$10$}
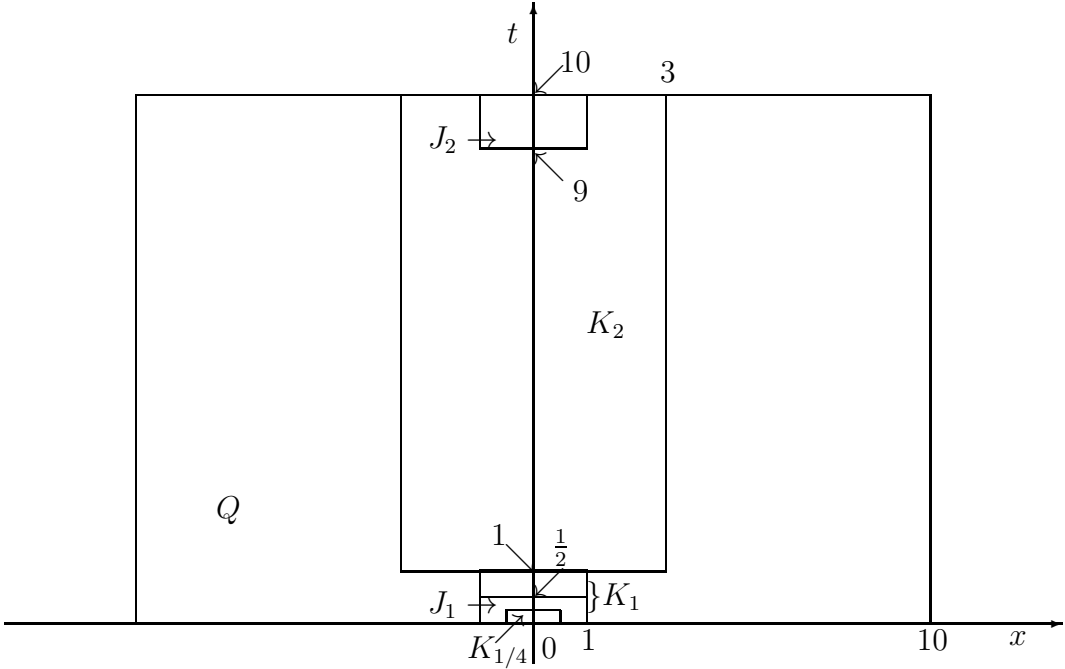
\put(95,16){Figure $1$. The cubes $J_1$, $J_2$, $K_1$, $K_2$, $K_{1/4}$.}
\end{picture}

%%%%%%%%%%%%%%     LEMMA  3.  2.   Barrier   0

\begin{lem}
\label{lem:barrier0}
There exist a nonnegative function $\phi\in %C(\ol{Q})\cap    CHANGED 
C^{2, 1}(\ol{Q})$ 
 and a function $g\in C(\ol{Q})$ such that    
\begin{eqnarray*}
\le\{ \begin{array}{rcll}
\phi_t +\P^+(D^2\phi)&\leq& g(x, t) &\mathrm{in} \ Q, \\ 
\phi&\geq&2 &\mathrm{in} \ K_2,\\ 
\phi&=&0 &\mathrm{on} \ \p_p Q,\\
\supp\,g&\subset& K_1.&
\end{array} \ri.
\end{eqnarray*}
\end{lem}

Letting $K_1$ as above, %:=(-r,r)^n\ti (0,r^2)$, 
we denote by $\C_1$ 
the set of all $2^{n+2}$ cubes $(-1+i_1,i_1)\ti\cdots 
\ti (-1+i_n,i_n)\ti 
(\fr{j}{4},\fr{j+1}{4}]$ for $i_k=0,1$ ($k=1,2,\ldots,n$), and $j=0,1,2,3$. 
For each cube $L\in\C_1$, we divide it into $2^{n+2}$ cubes. 
We denote by $\C_2$ the set of such cubes constructed by the same procedure from each cube $L\in\C_1$. 
Inductively, we construct $\C_k$ whose elements have length $2^{-k+1}$ in each space direction and $4^{-k}$ in time. 
We call $L\in \cup_{k=1}^\infty \C_k$ a dyadic cube of $K_1$. 
When $L\in\C_k$ is constructed from an element of $\C_{k-1}$ by the above procedure, we denote by $\tilde L\in\C_{k-1}$ the predecessor of $L$.

For $L\in\C_k$ and its predecessor $\tilde L:=J\ti (\tau ,\tau+\frac{1}{4^{k-1}}]\in \C_{k-1}$  
for a cube $J=(\frac{a_1}{2^{k-2}},\frac{a_1+1}{2^{k-2}})\ti \cdots \ti 
(\frac{a_n}{2^{k-2}},\frac{a_n+1}{2^{k-2}})$ with some integers 
$a_1,\ldots,a_n$, and  $\tau\in [0,1)$,   
we define  
\begin{equation}\label{referee1}\tilde{L}^m:= J\ti  \le(\tau +4^{-k+1} , 
\tau +4^{-k+1}(m+1)\ri],
\end{equation}
which is the union of $m$ cubes of the translated predecessor 
in the ``future'' direction. 
%Let us call $\tilde L$ and $\tilde L^m$ the father and $m$ uncles of $L$, respectively. 
%For any given cube $L=(x_0,t_0)+K_1$, %(-R,R)^n\ti (\alpha,\beta)$ (for $R>0,\alpha,\beta\in\R$), 
%we call $L$ a dyadic cube of $K$ if $L$ is constructed by this procedure. 

We define $\C:=\cup_{k=1}^\infty \C_k$. 
Moreover, for $m\in\N$, we define $\C (m):=\{ L\in \C \ | \ 
\tilde L^m\subset Q\}$. 
Notice that when $1\leq m\leq 36$, we have $\C(m)=\C$.

We recall a parabolic version of the Calder\'on-Zygmund decomposition, 
which is a modification of Lemma 2.4.27 of \cite{ImbSil13}.
Since fine cubes are needed in the proof of Lemma 2.4.27 of \cite{ImbSil13}, we can follow 
the argument there to prove the next lemma. 

%%%%%%%%%%%% 
%%%%%%%%%%%%     Parabolic cube decomposition      Lemma    3.    3. 
%%%%%%%%%%%%

\begin{lem}%(Lemma ??? in \cite{ImbSil13})
\label{CalZyg}
Let $m\geq1$ be an integer, and $K_1\subset\R^{n+1}$ be as above. 
Let measurable sets $A\subset B\subset K_1$ and $\s\in (0, 1)$ satisfy
\[\le\{\begin{array}{ll}
%\begin{description}
(i)&%\item[(a)] $
|A|\leq \s|K_1|,\\%$$;$
(ii)&%\item[(b)] 
\mbox{if }L\in \C(m)%\mbox{ is a dyadic cube of }K_1
\mbox{ is such that }
|A\cap L|> \s|L|,\mbox{ then }\tilde{L}^m\subset B,
\end{array}\ri.%\end{description}
\]
where $\tilde{L}^m$ is from $(\ref{referee1})$. 
Then, it follows that
\[
|A| \leq \s %(1-\theta)
\frac{m+1}{m}|B|.
\]
\end{lem}

%%%%%%%%%%%%%%%
%%%%%%%%%%%%%%%%%%%    Proof of weak Harnack, Theorem 3.1
%%%%%%%%%%%%%%%%%%%%

\begin{proof}[Proof of Theorem \ref{thm:weak0}]
%We recall  cubes $Q:=(-10,10)^n\ti (0,10]$, $J_1:=(-1, 1)^n\ti (0, \frac12 ]\subset K_1$, and $J_2:=(-1, 1)^n\ti (9, 10]$. 
%We will first consider the case $(\ref{Apq1})$. 
For $\e_1>0$, which will be fixed later, we set
\[\tilde{u}(x, t)=N_0u(x,t),\]
where $N_0=\le(\inf_{J_2} u +\e_1^{-1}\|f\|_{L^{p}(Q)} +\eta\ri)^{-1}$ 
for $\eta>0$, which will be sent to $0$ at the end of the proof. 
By considering $\tilde u$ instead of $u$, 
it is enough to show that there are $\e_0,C_0>0$ such that 
\begin{equation}\label{eq:integra}
\le(\int_{J_1} u^{\e_0} \,dxdt\ri)^{\fr{1}{\e_0}}\leq C_0 
\end{equation}
under the assumptions 
\begin{equation}\label{eq:Reduction}
\low{J_2} \inf \, u\leq 1,\quad\mbox{and}\quad \| f\|_{L^{p}(Q)}\leq \e_1.
\end{equation}

Let $\phi$ be the function in Lemma \ref{lem:barrier0}. 
By letting $w:=\phi-u$, it is immediate to see that $w$ is an $L^p$-viscosity subsolution of 
\[w_t +\P^-(D^2w)-\mu |Dw| -h=0 \quad \mbox{in } Q,%_1(0, 1).
\]
where $h:=\mu|D\phi|+g +f$. %=:h%We define $h:=g -f +\mu^+|D\phi|$. 
In view of Proposition \ref{thm:ABP}, we have  
\[\under{Q}%_1(0,1)}
{\sup}\, w\leq 
%\under{ \p_p \hat Q}%_1(0,1) }
%{ \sup }\, w^+ +C%\exp\le(C\|\mu^+\|_{L^{n+1}(Q))}\ri)
C\|h\|_{L^{p}(Q_{+}[w])}.
\] 
Hence, by recalling supp $g\subset K_1$ in Lemma 
\ref{lem:barrier0}, 
it is easy to verify that this inequality implies 
\[%\begin{equation}\label{eq:UseABP}
1\leq \low{J_2}\sup \, w\leq C\le(\|g\|_{L^{p}( Q_{+}[w])}+\e_1+\d_0\|D\phi\|_{L^\infty(Q)}\ri) .
\] 
Thus, since $\phi\in C^{2,1}(\overline Q)$, for  fixed $\e_1, \d_0>0$, there is $\theta\in (0,1)$ 
such that $|\{ (x,t)\in K_1 \ | \ w(x,t)>0\}| \geq \theta |K_1|$. 
Hence, setting $M:=\under{K_1}{\sup}\,\phi$, ($M\geq 2$),  we have 
\begin{equation}\label{eq:basic}
|\{(x, t)\in K_1\ |\ u(x, t)\geq  M\}|\leq (1-\theta)|K_1|.
\end{equation}

%Under assumption (\ref{Apq2}),  since Proposition \ref{thm:ABP} $(ii)$ yields (\ref{eq:basic}) 
%by replacing $n+1$ by $p$, we can obtain the above estimate for some $\theta\in (0,1)$. 

%%%%%%%%%%%%%%%
%%%%%%%%%%%%%%%   Notation   m  \delta  S    C_1   k_0   C_2  
%%%%%%%%%%%%%%%

We next fix $\delta\in (1-\theta,1)$ 
and select large $m\in\N$ such that 
\begin{equation}\label{eq:Fixm}
1-\theta<(1-\theta)\fr{m+1}{m}\leq \delta<1.%,\quad 0<\fr{1}{m}\leq \delta, \quad\mbox{and}\quad \fr{1}{2\sqrt{m+1}}\leq \fr13.
\end{equation}
Letting $J_1^{k}:=(-1,1)^n\ti (0,\fr{1}{2}+
\fr{m+1}{9^{mk-3}(9^m-1)}]$ for $k\geq 1$, 
we note that 
\[J_1^{k+1}\subset J_1^k\quad (k\in\N),\quad \mbox{and}\quad 
\lim_{k\to\infty}J_1^k=J_1 .\]
We choose $k_0\in\N$ such that
\[
\fr{m+1}{9^{mk_0-3}(9^m-1)}<\fr{1}{2}\quad 
(i.e. \quad J_1^k\subset K_1\quad \mbox{for } k\geq k_0).
\]
Finally, putting 
$\hat C_0:=|J_1^{k_0}|\le(m(1-\theta)^{-1}(m+1)^{-1}\ri)^{k_0},
$
by our choice of $\d$ (i.e.  \eqref{eq:Fixm}), we observe
\begin{equation}\label{3k0}
|J_1^{k_0}|\leq \hat C_0\delta^{k_0}.
\end{equation} 

We will show that 
%that there exists a constant $C_2=C_2(n,\L/\l, q)>0$ such that   
\begin{equation}\label{eq:dis}
|\{(x, t)\in J_1^k\ |\ u(x, t)\geq M^{km} \}|\leq \hat C_0\delta^{k} \quad (\forall k\geq k_0).
\end{equation}
Notice that (\ref{3k0}) yields (\ref{eq:dis}) for $k=k_0$. 

%%%%%%%%%%%%%%%%%
%%%%%%%%%%%%%%%%%    Check hypotheses in Lemma 
%%%%%%%%%%%%%%%%%

For any fixed $k\geq k_0+1$, 
we suppose that $(\ref{eq:dis})$ holds for $k-1$. 
Set %Let us define the sets 
\[
A=\{(x, t)\in J_1^{k}\ |\ u(x, t)\geq M^{km} \} \mbox{ and } 
B=\{(x, t)\in J_1^{k-1} \ |\ u(x, t)\geq M^{(k-1)m} \}. 
\]
It is immediate to see that $A\subset B\subset K_1$, 
and $|A|\leq\le(1-\theta\ri)|K_1|$ from $(\ref{eq:basic})$ because 
$A\subset \{ (x, t)\in K_1 \ | \ u(x,t)\geq M\}$. 

If the hypotheses in Lemma \ref{CalZyg} are satisfied for $A$, $B$ and $\s=1-\theta$, then using $|B|\leq \hat C_0\delta^{k-1}$, we have 
\[|A|\leq(1-\theta)\fr{m+1}{m}\hat C_0\delta^{k-1}.\]
Hence, (\ref{eq:dis}) holds for any $k\geq k_0$ by our choice of $\delta$ and $m$ in \eqref{eq:Fixm}. 
Therefore, the standard argument implies 
\begin{equation}\label{Decay}
|\{ (x,t)\in J_1 \ | \ u(x,t)\geq s\}|
\leq A_0s^{-\b_0}\quad (s>0),
\end{equation}
where $A_0=\hat C_0\d^{-1}$ and $\b_0:=-\fr{\log \d}{m\log M}>0$. %%%%%%%% CHANGED 
We thus obtain \eqref{eq:integra} when $\e_0\in (0,\b_0)$.

In order to check $(ii)$ in Lemma \ref{CalZyg}, we take a dyadic cube 
$L\in \C (m)$ such that 
\begin{equation}\label{eq:AL}
|A\cap L|> (1-\theta)|L|.
\end{equation}
We can find $j\in\N$ and $(x_0,t_0)\in \ol K_1$ such that
\[L=(x_0, t_0)+\le( -2^{-j} , 2^{-j}\ri)^n\ti\le(0, 2^{-2j}\ri]. 
\]
%%%%%%%%%%%%%%%%      CLAIM 

We claim that if \eqref{eq:AL} holds 
%$|\{ (x,t)\in L \ | \ u(x,t)\geq  M^{km}\}|> (1-\theta)|L|$ for $M>1$, 
then 
\begin{equation}\label{claim}
\low{N_\ell\cap \{\R^n\ti (0,10]\} }\inf u%\{ u(x,t) \ | \ (x,t)\in L_\ell, t\in (0, 1)\}
> M^{km-\ell}\geq 1\quad\mbox{for }
\ell\in\{ 1, 2, \cdots, km\}.
\end{equation}
Here, we set 
$%\begin{align*}
N_1:=(x_0, t_0+\fr{1}{4^{j}})+\fr{3}{2^{j}}K_1, \cdots, 
N_\ell:=(x_0, t_0+\fr{9^\ell-1}{8}\cdot\fr{1}{4^{j}})+\fr{3^\ell}{2^{j}}K_1,\cdots$ (See Figure 2),  
%$%\end{align*}
where for $\s>0$, 
\[\s K_1:=(-\s, \s)^n\ti (0,\s^2].\]
We will prove this claim later. 
\vspace{1.5cm}
%%%%%%%%%%%%%%%%%%%%%%       Fig  2

\begin{picture}(450,250)
\put(0,50){\vector(1,0){400}}%\UTF{00C2}\UTF{0089}\UTF{00C2}\UTF{00A1}\UTF{00C2}\UTF{008E}\UTF{00C2}\UTF{00B2}%
\put(200,35){\vector(0,1){250}}%\UTF{00C2}\UTF{008F}c\UTF{00C2}\UTF{008E}\UTF{00C2}\UTF{00B2}%
%\put(50,50){\framebox(300,200)}%Q

\put(197,50){\framebox(6,2)}% L
\put(191,52){\framebox(18,12)}% L1
\put(173,64){\framebox(54,72)}% L2
\put(119,136){\line(1,0){162}}% L3
\put(119,136){\line(0,1){140}}
\put(281,136){\line(0,1){140}}
\put(203,37){$(x_0,t_0)$}
\put(350,37){$x$}
%\qbezier(140,276)(0,-180)(260,276)
%\put(248,255){$3$}
\put(190,270){$t$}
%\put(345,40){$10$}
%\put(200,253){$\swarrow$}
\put(175,41){$L\nearrow$}
\put(205,53){$\leftarrow N_1$}
\put(210,100){$N_2$}
\put(230,200){$N_3$}
\put(120,10){Figure $2$. The cubes $L$, $N_\ell$.}
\end{picture}

One direct consequence of this claim for $\ell=1,2,\ldots, m$ is the following assertion: under \eqref{eq:AL}, it follows that  
\begin{equation}\label{olLmSubset}
u>M^{(k-1)m}\quad \mbox{in }\bigcup_{\ell=1}^m N_\ell \cap \{\R^n\ti (0,10]\}.
\end{equation}

%%%%%%%%%%%%%%%%%
It is obvious from the definition that    
%\begin{lem}
\begin{equation}\label{lem:copy}%(Corollary 2.4.26 in \cite{ImbSil13})
\widetilde{L}^k\subset \G_k\quad \mbox{for }k\in\N,
\end{equation}
where $
\G_k:=\bigcup_{\ell=1}^kN_\ell$ 
%\over{k}{\low{\ell=1}{\bigcup}} \, L_\ell$ 
for $k\in\N$. 
We also write $\G_\infty =\bigcup_{\ell\in\N}N_\ell
$.

We easily verify the following inclusions: 
\[%\label{parabolas}
(x_0,t_0)+S^-_{\fr{1}{4^j},%\fr{1}{2^j},
\infty}\subset \G_\infty%\bigcup_{\ell=1}^\infty L_\ell
\subset 
(x_0,t_0)+S^+_{\fr{1}{4^j},\infty},
\]
where for $0\leq \a<\b\leq\infty$, paraboloid type domains $S^\pm_{\a,\b} $ are given by
\[
S^-_{\a,\b}:=\le\{ (x, t)\in\R^n\ti (\a,\b] \ \le| \ t>2^{-3} (
9|x|_\infty^2-4^{-j})\ri.\ri\},
\]
\[
S^+_{\a,\b}:=
\le\{ (x, t)\in\R^n\ti (\a,\b] \ \le| \ t>2^{-3} (
|x|_\infty^2-4^{-j})\ri.\ri\}.
\]
Here, $|x|_\infty :=\max\{ |x_1|,\ldots ,|x_n|\}$ for $x=(x_1,\ldots,x_n)\in\R^n$. 
For extreme cases when $(x_0,t_0)=(\hat x,0)$ or $(x_0,t_0)=(\hat x,1)$, where $\hat x=(1,0,\ldots,0)$, 
we observe \[
J_2\subset (\hat x,1)+S^-_{0,9}\quad\mbox{and}\quad%(-s,s)^n\ti\{ 2^{-3}(9s^2-2^{-2j})\}
 (\hat x,0)+S^+_{0,%\fr{3}{2}, 
 10}%(-s,s)^n\ti \{ 2^{-3}(s^2-2^{-2j})\}
\subset Q. 
\]
%where $S^\pm_{\a,\b}:=S^\pm_\a\cap \{ \R^n \ti (0,\b]\}$, 
Hence, we obtain 
\begin{equation}\label{EqGinQ}
J_2\subset \G_\infty%\le(\overset{\infty}{\under{\ell=1}{\bigcup}} L_\ell\ri)
\cap 
\{\R^n\ti(0, 10]\}
\subset Q.
\end{equation}
%We refer to Lemma 2.4.23 in \cite{ImbSil13} for its proof. 

Now, assuming (\ref{eq:AL}), we will prove 
$\tilde{L}^m\subset B$. 
To this end, by \eqref{olLmSubset} and \eqref{lem:copy}, 
it is enough to show that
\[
\tilde L^m\subset J^{k-1}_1.
\]

On the other hand, since $(\ref{eq:AL})$ yields 
\[|J_1^{k}\cap L|> (1-\theta)|L|> 0,
\]
we have $(-1,1)^n\ti ( 0,\frac{1}{2}+\frac{m+1}{9^{mk-3}(9^m-1)}
]\cap  L\neq \emptyset$. 
By the definition of $\widetilde{L}^m$, we have
\begin{equation}
\label{eq:Lm3}
\widetilde{L}^m\subset (-1, 1)^n\ti\le(0, \fr{1}{2} +\fr{m+1}{9^{mk-3}(9^m-1)} +
\fr{m+1}{4^{j-1}}\ri].
\end{equation}
Setting $\ell^*=\min\{ k\in\N \ | \ L_{k+1}\cap \R^n\ti(0, 10]= \emptyset\}$, %from Lemma \ref{lem:stacks},
 we have 
\[
J_2\subset \G_{\ell^\ast}%\le(\over{\ell^\ast}{\low{\ell=1}{\bigcup}} L_\ell\ri)
\cap\{\R^n\ti(0, 10]\}. 
\]
Since $\low{J_2}{\inf}\, u\leq 1$, by \eqref{claim} again for $\ell=1,2,\ldots,km$, 
we thus have  
\[
\low{\G_{km}%\le(\overset{km}{\low{\ell=1}{\bigcup}} L_\ell\ri)
\cap \{\R^n\ti(0, 10]\}}{\inf}\, u > 1\geq 
\low{\G_{\ell^\ast}%\le(\overset{\ell^*}{\low{\ell=1}{\bigcup}} L_\ell\ri)
\cap \{\R^n\ti(0, 10]\}}{\inf}\, u
\] 
which implies $km< \ell^\ast$. 
Hence, noting
\[t_0 +2^{-2j-3}\le( 9^{\ell^\ast}-1\ri)\leq 10,  
\] 
we have  
\[
2^{-2j}\leq  \fr{80}{9^{km}-1}
\] 
which, together with  $(\ref{eq:Lm3})$, yields
\[
\tilde{L}^m\subset (-1, 1)^n\ti\le(0, \fr{1}{2} +\fr{m+1}{9^{mk-3}(9^m-1)} +
\frac{320(m+1)}{9^{mk}-1}\ri].
\]
Therefore, noting
\[
\fr{1}{9^{mk-3}(9^m-1)}+\fr{320}{9^{mk}-1}\leq \fr{1}{9^{m(k-1)-3}(9^m-1)}
,\]
we can apply Lemma \ref{CalZyg} to conclude the proof.

%%%%%%%%%%%%%
%%%%%%%%%%%%%     proof of     claim
%%%%%%%%%%%%%

It remains to show that \eqref{claim} holds under \eqref{eq:AL}. 

By setting $v(x, t)= M^{1-km}u(x_0 +\fr{1}{2^j} x, t_0 +\fr{1}{4^j}t)$,  
 \eqref{eq:AL} implies 
\begin{equation}
\label{eq:contra}
|\{(x,t)\in K_1 \ | \ v(x,t)\geq M \}|> (1-\theta)|K_1|. 
\end{equation}
However, we note that $v$ is a nonnegative $L^p$-viscosity supersolution of  
\[v_t +\P^+(D^2v) +\widetilde{\mu}|Dv|-\widetilde{f}=0\quad\mbox{in }Q, 
\]
where $\tilde{\mu}(x, t)= \fr{1}{2^j}\mu(x_0 +\fr{1}{2^j} x, t_0 +\fr{1}{4^j}t)$, $\tilde{f}(x, t)=\fr{1}{M^{km-1}4^j}f(x_0 +\fr{1}{2^j} x, t_0 +\fr{1}{4^j}t)$. 
We notice that 
\[\|\tilde{f}\|_{L^{p}(Q)}\leq\e_1, \quad\mbox{and}\quad\le\|\tilde{\mu}\ri\|_{L^q(Q)} \leq \|\mu\|_{L^q(Q)}
\]
because $q>n+2$ and $p>\fr{n+2}{2}$.  
Thus, if $\under{K_2}{\inf}\,v\leq1$ holds, then  the same argument to obtain (\ref{eq:basic}) yields 
\begin{equation}\label{eq:contra1}
|\{(x,t)\in K_1 \ | \ v(x,t)\geq M \}|\leq (1-\theta)|K_1|,
\end{equation} 
which contradicts (\ref{eq:contra}). 
Hence, we have $v>1$ in $K_2$, namely, (\ref{claim}) holds for $\ell=1$ 
by the definition.

Next, for $\ell\geq2$, we suppose that (\ref{claim}) holds for $\ell-1$. 
We may suppose that $N_{\ell-1}\subset \R^n\ti(0, 10]$ since otherwise 
$N_\ell \cap \{\R^n\ti (0,10]\}=\emptyset$, which concludes \eqref{claim} for $\ell$. %we are done. 
Thus,  since   
\[\under{N_{\ell-1}}{\inf}\, u= \low{N_{\ell-1}\cap \{\R^n\ti(0, 10]\}}{\inf}\,u> M^{km-\ell+1}, 
\]
we have a trivial inequality
\begin{equation}\label{eq:L_ell-1}
\le|\le\{ (x,t)\in N_{\ell-1} \ | \ u(x,t)\geq M^{km-\ell+1}\ri\}\ri|= |N_{\ell-1}|> (1-\theta)|N_{\ell-1}|.
\end{equation}

Set  $w(x, t):= \fr{1}{M^{km-\ell}}u(x_0 +\fr{3^{\ell-1}}{2^{j}} x, t_0 +\fr{9^{\ell-1}-1}{8\cdot 4^{j}} +\fr{9^{\ell-1}}{4^{j}}t)$. 
In view of \eqref{EqGinQ}, we easily see that 
\[
\le(x_0,t_0+\fr{9^{\ell-1}-1}{8\cdot 4^j}\ri)+\le(-\fr{10\cdot 3^{\ell-1}}{2^j}, 
\fr{10\cdot 3^{\ell-1}}{2^j}\ri)^n\ti \le(0,\fr{10\cdot 9^{\ell-1}}{4^j}\ri]\subset Q.\]
Hence, it  follows that $w>1$ in $K_2$ because, 
if $\inf_{K_2}w\leq 1$, then the above argument again implies \eqref{eq:contra1} for $w$ in place of $v$,  
which contradicts \eqref{eq:L_ell-1} for $w$. 
\end{proof}

%%%%%%%%%
%%%%%%%%%%\UTF{00C2}\UTF{0081}@\UTF{00C2}\UTF{0081}@\UTF{00C2}\UTF{0081}@\UTF{00C2}\UTF{0081}@\UTF{00C2}\UTF{0081}@\UTF{00C2}\UTF{0081}@\UTF{00C2}\UTF{0081}@\UTF{00C2}\UTF{0081}@\UTF{00C2}\UTF{0081}@\UTF{00C2}\UTF{0081}@3. 2. A general case
%%%%%%%%%%%%%
%%%%%%%%%%%%%%%%

\subsection{A general case}

In order to show the weak Harnack inequality without assuming \eqref{wHsmall},  
we use a new barrier function, which will be constructed in 
Lemma \ref{lem:barrier}. 
%We first show that the solutions 
%constructed in Proposition \ref{prop:exist} admit  
%the weak Harnack inequality in small regions. 
%We note that the inequality in \eqref{Exists} cannot be replaced by 
%the equality because we do not have enough regularity for approximate 
%solutions.  
%Therefore, the next lemma is not trivial.  

%%%%%%%%%%%%%%%%%%%%%%%%%
%%%%%%%%%%%%%%%%%%%%%%%%%     Theorem  3. 4. 
%%%%%%%%%%%%%%%%%%%%%%%%%

\begin{thm}
\label{thm:weak}
Let \eqref{Apq1} hold, $f\in L^p_+(Q)$ and $\mu\in L^q_+(Q)$.  
There exist 
$\e_0=\e_0(n, \L, \l, p, q$, $\|\mu\|_{L^q(Q)})>0$ and 
$C_0=C_0(n,\L,\l,p,q,\|\mu\|_{L^q(Q)})>0$ %, and $s=s(n,\L/\l,p,q,\|\mu\|_{L^q(Q)})\in(0, 1)$ 
such that 
 any nonnegative $L^p$-viscosity supersolution $u$ of $(\ref{P+})$ satisfies   
\[%\label{WHI1}
\le(\int_{J_1} u^{\e_0} \,dxdt\ri)^{\fr{1}{\e_0}}\leq C_0\le(\low{J_2}{\inf}\, u +\|f\|_{L^p(Q)}\ri).
\]
\end{thm}

%%%%%%%%%%%%%%%%     REMARK   3.  5. 

\begin{rem}
The constants $\e_0,C_0$ above depend on 
$\| \mu\|_{L^q(Q)}$ in a sense that even 
if we consider a different $\hat \mu\in L^q(Q)$ such that $\|\hat\mu\|_{L^q(Q)}\leq 
\|\mu\|_{L^q(Q)}$ in place of $\mu$ in Theorem \ref{thm:weak}, 
the same conclusion holds true with the same constants as in Theorem \ref{thm:weak}. 
\end{rem}

%%%%%%%%%%%%%%    Remark 3. 6.

\begin{rem}
When $\mu\in L^\infty (Q)$,  $\phi$ in the next lemma can be given by a modified heat kernel from \cite{W}. 
However, since we have unbounded $\mu$, 
it is not possible to construct such a precise function for $\phi$ below. 
\end{rem}

%%%%%%%%%%%%%%     LEMMA   3.  7.  Barrier

\begin{lem}
\label{lem:barrier}
Let $q>n+2$ and $\mu\in L^q_+(Q)$. There exist a nonnegative function $\phi\in C(\ol{Q})\cap W^{2, 1}_{q}(Q)$ 
 and a function $g\in L^q(Q)$ such that    
\[
\le\{ \begin{array}{rcll}
\phi_t +\P^+(D^2\phi)+\mu (x,t)|D\phi|&\leq& g(x, t) &a.e. \ \mathrm{ in} \ Q, \\ 
\phi&\geq&2 &\mathrm{in} \ K_2,\\ 
\phi&=&0 &\mathrm{on} \ \p_p Q,\\
\supp\,g&\subset& K_1.&
\end{array} \ri.
\]
\end{lem}

%%%%%%%%%%%%%%%%%
%%%%%%%%%%%%%%%%%     Proof of Lemma    barrier
%%%%%%%%%%%%%%%%%

\begin{proof}
Choose a nonnegative function $\xi\in C^\infty (\ol{Q})$ such that $\xi =0$ in 
$\ol{Q}\setminus K_{1/4}$, where $K_{1/4}:=(-\frac 12,\frac 12)^n\ti (0,\frac{1}{4}]$ (see Fig $1$), and 
$\xi (x,0)>0$ for $x\in (-\frac 12,\frac 12)^n$. 
In view of Proposition \ref{prop:exist}, 
we can find a nonnegative function  
$\psi\in C(\ol{Q})\cap W^{2,1}_{q}(Q)$ satisfying 
\[
\le\{\begin{array}{rl}
\psi_t+\P^+(D^2\psi)+\mu |D\psi|= 0&a.e.\mbox{ in }Q,\\
\psi =\xi&\mbox{on }\p_pQ.
\end{array}\ri.
\]

We claim that there exists $\s>0$ such that 
\[
\psi\geq \s\quad\mbox{in }K_2.
\]
In fact, assuming $\psi (x_0,t_0)=0$ for $(x_0,t_0)\in \ol K_2$, we will 
obtain a contradiction. 

For $r\in (0,\fr{1}{\sqrt{10}}]$, 
we set $v_0(x,t)=\psi (x_0+rx,t_0+r^2(t-10))=0$ for $(x,t)\in Q$, 
$v_0(0,10)=0$ and the function $v_0$ is a solution 
of 
\[
(v_0)_t+\P^+(D^2v_0)+\hat \mu|Dv_0|=0\quad\mbox{in }Q,\]
where 
\[
\hat\mu(x,t)=r\mu (x_0+rx,t_0+r^2(t-10)).\]
Since  it follows that
\[
\| \hat\mu\|_{L^{p}(Q)}\leq 
r^{1-\fr{n+2}{q}}\| \mu\|_{L^q(Q)},
\]
if we choose $r:=\le(\d_0\|\mu\|_{L^{p}(Q)}^{-1}\ri)^{\fr{q}{q-(n+2)}}$, where $\d_0$ is from
Theorem \ref{thm:weak0} for $p=q$, then Theorem \ref{thm:weak0}
yields $v_0=0$ in $J_1$. To continue the proof we will assume (without loss of generality) that $t_0=1$.
If $x_0\in [-\fr14,\fr14]^n$, then Theorem \ref{thm:weak0} 
implies $\psi (x_0,0)=0$, which contradicts our choice of $\psi$. 
Thus, without loss of generality, it is enough to consider $x_0\in \p (-3,3)^n$. 
Therefore, we can choose $x_1 \in (-3,3)^n$ such that 
$x_0\in  x_1+\p (-r,r)^n$.

\vspace{1cm}

%%%%%%%%%%%%%%%%%%%%%%%%%
%%%%%%%%%%%%%%%%%%%%%%%%%    Fig   3
%%%%%%%%%%%%%%%%%%%%%%%%%

\begin{picture}(420,260)
\linethickness{1pt}
\put(0,10){\vector(1,0){420}}%\UTF{00C2}\UTF{0089}\UTF{00C2}\UTF{00A1}\UTF{00C2}\UTF{008E}\UTF{00C2}\UTF{00B2}%
\put(210,0){\vector(0,1){270}}%\UTF{00C2}\UTF{008F}c\UTF{00C2}\UTF{008E}\UTF{00C2}\UTF{00B2}%
\put(10,210){\line(0,1){50}}%{\framebox(400,200)}%K_1
\put(410,210){\line(0,1){50}}
\put(10,210){\line(1,0){400}}

\put(175,10){\framebox(70,50)}

\put(15,225){$(x_0,1)$}
\put(22,222){\vector(-1,-1){10}}

\put(-18,200){$10r_k^2\{$}
\put(9,190){$\smile$ \hspace{-2mm}$\smile$}
\put(12,176){$r_k$}
\put(11,185){$\uparrow$ \hspace{-2mm}$\nearrow$}
%\put(8,215){$\downarrow$}
\put(8,208){$\bullet$}
\put(18,208){$\bullet$}
\put(27,192){$\bullet$}
\put(47,176){$\bullet$}
\put(67,161){$\bullet$}
\put(172,76){$\bullet$}
\put(153,92){$\bullet$}

\put(40,215){$ (x_1,1)$}
\put(40,220){\vector(-2,-1){15}}

\put(180,-2){$(0,0)$}
\put(405,0){$3$}
\put(415,0){$x$}
\put(350,220){$K_2$}
\put(220,30){$K_{1/4}$}
\put(10,195){\framebox(20,15)}
\put(30,180){\framebox(20,15)}
\put(50,165){\framebox(20,15)}
\put(155,80){\framebox(20,15)}
\put(10,195){\framebox(20,2)}
\put(10,206){\framebox(20,4)}
%\put(175,30){\dashline(0,1){30}}
\put(50,205){\vector(-2,-1){18}}
\put(38,138){\vector(1,4){10}}
\put(185,113){\vector(-1,-4){8}}
\put(170,120){$(x_k,(1-10(k-1)r_k^2)$}
\put(75,150){$\ddots$}
\put(140,100){$\ddots$}
\put(212,65){$\fr14$}
\put(240,0){$\fr12$}
\put(50,200){$(x_2,1-10r_k^2)$}
\put(15,130){$(x_3,1-20r_k^2)$}
\put(200,220){$t$}
\put(215,215){$1$}
\put(150,-18){Figure $3$. The procedure.}

\linethickness{0.1pt}
\put(175,60){\line(0,1){20}}
\put(10,10){\line(0,1){200}}
\put(410,10){\line(0,1){200}}
\end{picture}

\vspace{1.5cm}

%%%%%%%%%%%%%%%%%%%%%%%%%%
%%%%%%%%%%%%%%%%%%%%%%%%%

Setting $r_k=\fr{5}{2(k-1)}$ for $k\geq 1+\fr{5}{2r}$ (i.e. $r_k\leq r$), %$k\geq 2$, 
if we fix $k\geq \max\{ \fr{253}{3}, 1+\fr{5}{2r}\}$, then 
\[
10r_k^2(k-1)\leq  \fr34 .\]
Thus, using Theorem \ref{thm:weak0} finitely many times,  we can find $(x_k,1-10(k-1)r_k^2)\in 
[-\fr12,\fr12]^n\ti [\fr14,1]$ such that $u(x_k,1-10(k-1)r_k^2)=0$. 
See Fig 3 for this procedure. 
Hence, by Theorem \ref{thm:weak0} again, we arrive at a contradiction.

Therefore, %%%%%%% CHANGED 
for a large number $\hat M>0$, we verify that 
$\hat M\psi \geq 2$ in $K_2$. 
Now, let $\eta\in C^\infty (\ol Q)$ be a nonnegative function 
such that 
\[
\eta =1\quad \mbox{in }Q\setminus K_1,\quad\mbox{and}\quad 
\eta =0\quad\mbox{in }K_{1/4}.\]
It is easy to observe that $\phi:=\hat M\eta\psi$ satisfies the desired properties. 
In fact, we may choose 
$g=\hat M[\psi\eta_t+\P^+(\psi D^2\eta+2D\eta\otimes D\psi)+\mu\psi|D\eta|]$. 
\end{proof}

%%%%%%%%%%%%%      Remark  3.  8.  

\begin{rem}
We notice that the global $W^{2,1}_p(Q)$ estimate of Proposition \ref{prop:exist} is necessary to verify that $g\in L^p(Q)$ in the final step of the above proof. 
\end{rem}

\begin{proof}[Proof of Theorem \ref{thm:weak}]
%We recall  cubes $Q:=(-10,10)^n\ti (0,10]$, $J_1:=(-1, 1)^n\ti (0, \frac12 ]\subset K_1$, and $J_2:=(-1, 1)^n\ti (9, 10]$. 
%We only consider the case $(\ref{Apq1})$ for $p,q,n$ 
%since the other case can be proved in a similar manner. 
For $\e_1>0$, which will be fixed later, we set
\[\tilde{u}(x, t)=N_0u(x,t),\]
where $N_0=\le(\inf_{J_2} u +\e_1^{-1}\|f\|_{L^{p}(Q)} +\eta\ri)^{-1}$ 
for $\eta>0$, which will be sent to $0$ at the end of the proof. 
As in the proof of Theorem \ref{thm:weak0},  
it is enough to show that there are $\e_0,C_0>0$ such that 
\eqref{eq:integra} holds under assumptions \eqref{eq:Reduction}.
%$$\low{J_2} \inf \, u\leq 1,\quad\mbox{and}\quad \| f\|_{L^{n+1}(Q)}\leq \e_1.$$

Let $\phi$ be the function from Lemma \ref{lem:barrier}. 
By letting $w:=\phi-u$, it is immediate to see that $w$ is an $L^p$-viscosity subsolution of 
\[w_t +\P^-(D^2w) -\mu |Dw|-h=0 \quad \mbox{in } Q,%_1(0, 1).
\]
where $h:=g -f$. %=:h%We define $h:=g -f +\mu^+|D\phi|$. 
In view of Proposition \ref{thm:ABP}, we have  
\[\under{Q}%_1(0,1)}
{\sup}\, w\leq 
%\under{ \p_p \hat Q}%_1(0,1) }
%{ \sup }\, w^+ +C%\exp\le(C\|\mu^+\|_{L^{n+1}(Q))}\ri)
C\|h\|_{L^{p}%n+1}
(Q_{+}[w])}.
\]
Hence, it is easy to verify that this inequality implies 
\[%\begin{equation}\label{eq:UseABP}
1\leq \low{J_2}\sup \, w\leq C\|h\|_{L^{p}( Q_{+}[w])} .
\]%\end{equation}
Recalling that supp $g\subset K_1$ in Lemma 
\ref{lem:barrier},
we can find $\hat C=\hat C(n,\L,\l,p,q,\|\mu\|_{L^q(Q)})>0$ such that
\[
%\leq  C\|h\|_{L^{n+1}(\hat Q_{+}[w])} \\
1\leq \hat C\le(\|g\|_{L^{p}( Q_{+}[w]\cap K_1)}+\e_1\ri).
%\leq  \hat C
%\le(\| g\|_{L^p(Q)}| Q_+[w]\cap K_1|^{\fr{1}{n+1}-\fr{1}{p}}+\e_1\ri).
\]
%where $Q_{1+}(0, 1)=\{ w>0 \}$. 
Thus, for some fixed $\e_1>0$, there is $\theta\in (0,1)$ 
such that $|\{ (x,t)\in K_1 \ | \ w(x,t)>0\}| \geq \theta |K_1|$. 
Hence, as before, we obtain \eqref{eq:basic}. 

We can follow the same arguments as those in the proof of Theorem \ref{thm:weak0} 
to conclude the proof. 
\end{proof}

%%%%%%%%%%%   REMARK  3. 9. 

\begin{rem}\label{RwHI}In the above proof, we have shown that there exist $A_0,\b_0,\e_1>0$ such that 
if $u\in C(Q)$ is an $L^p$-viscosity supersolution of \eqref{P+} satisfying \[\low{J_2}\inf \, u\leq 1,
\]and if $\| f\|_{L^p(Q)}\leq \e_1$, 
then \eqref{Decay} holds true. 
%, it follows that $$|\{ (x,t)\in J_1 \ | \ u(x,t)\geq s\}|\leq A_0s^{-\b_0}\quad \mbox{for }s\geq 1.$$
\end{rem}

%%%%%%%%%%%%%%%%%%%%%%%%%%%%%%%%%%%%%%%%%%%%%%%%%%%%%%%%%%
%%%%%%%%%%%%%%
%%%%%%%%%%%%%%%%%%%%%      Section  4   Applications
%%%%%%%%%%%%%%%%%%%%%
%%%%%%%%%%%%%%%%%%%%%%%%%%%%%%%%%%%%%%%%%%%%%%%%%%%%%%%%%%%%

\section{Applications}

In this section, we consider 
 $L^p$-viscosity solutions of 
\begin{equation}\label{4-1}
u_t+G(x, t, Du, D^2u)-f(x, t)=0 \quad\mbox{in }Q,
\end{equation}
where $Q=(-10,10)^n\ti (0,10]$, and 
$G:Q\ti \R^n\ti S^n\to\R$ and $f:Q\to\R$ are given. 
We assume the following hypotheses for $G$ and $f$: 
%\begin{equation}\label{4-2}
%G(x, t, 0, O)=0\quad\mbox{for }(x, t)\in Q,\end{equation}
%\begin{equation}\label{4-3}
%\le\{\begin{array}{c}
%\P^-(X-Y)\leq G(x,t,\xi,X)- G(x,t,\xi,Y)\leq \P^+(X-Y)\\
%\mbox{for }(x, t, \xi, X, Y)\in Q\ti\R^n\ti S^n\ti S^n,
%\end{array}\ri.
%\end{equation}
\begin{equation}\label{4-4}
\le\{
\begin{array}{c}
\mbox{there exists } \mu\in L^q(Q) \mbox{ for } q>n+2 \mbox{ such that}\\
|G(x, t, \xi, O)|\leq\mu(x, t)|\xi| \mbox{ for } (x, t)\in Q\mbox{ and } \xi\in \R^n,\end{array}\ri. 
\end{equation}
\begin{equation}\label{4-5}
f\in L^p(Q) \quad\mbox{for } p\in (p_1,q].
\end{equation}

\begin{rem}
We note that \eqref{4-4} yields
\[
G(x,t,0,O)=0\quad\mbox{for }(x,t)\in Q.\]
Under \eqref{4-4} and \eqref{4-5}, if we suppose that $G$ satisfies \eqref{UP}, 
then it is easy to observe that if $u\in C(Q)$ is an 
$L^p$-viscosity subsolution (resp., supersolution) of \eqref{4-1}, then 
it is an $L^p$-viscosity subsolution (resp., supersolution) of 
\begin{equation}\label{eq:exteqs}
u_t+\P^-(D^2u)-\mu |Du|-f=0\quad
\le(\mbox{resp., } \ u_t+\P^+(D^2u)+\mu |Du|-f=0\ri)\quad\mbox{in }Q.
\end{equation}
Thus the properties of $L^p$-viscosity solutions of \eqref{4-1} discussed in this section will follow from the properties of 
$L^p$-viscosity sub/supersolutions of the extremal equations \eqref{eq:exteqs}.
\end{rem}

\subsection{H\"older continuity}

We show that  
the weak Harnack inequality for $L^p$-viscosity supersolutions of \eqref{P+} yields the H\"older continuity of solutions of \eqref{4-1} under the above hypotheses. This was remarked in \cite{KS3} for elliptic PDE.

For $r\in (0,1)$, we set
\[
Q_r:=(-10r,10r)^n\ti (10-10r^2,10] .\]
Notice that $Q_{10r}(0,10)$ defined in Section 2 is slightly different from this $Q_r$. 

%%%%%%%%%%%%%%%    Theorem  4.  2. 

\begin{thm}\label{thm:Holder}
Let $G$ satisfy \eqref{UP}, \eqref{4-4} and \eqref{4-5}. 
There exist $C>0$ and $\alpha\in (0,1)$ such that if $u\in C(Q)$ is an $L^p$-viscosity solution of \eqref{4-1}, then 
\[
|u(x,t)-u(\hat x,\hat t)|\leq C\le( |x-\hat x|^2+|t-\hat t|\ri)^{\fr{\a}{2}}
(\|u\|_{L^\infty(Q)}+\|f\|_{L^p(Q)})\quad\mbox{for }(x,t),(\hat x,\hat t)\in Q_{\fr{1}{2}}.
\]
\end{thm}

%One direct consequence of this claim for $\ell=m$ is the following assertion: under \eqref{eq:AL}, it follows that  
%\begin{equation}\label{olLmSubset}
%u>M^{m(k-1)}\quad \mbox{in }N_m\cap \{\R^n\ti (0,10]\}.
%\end{equation}

\begin{proof}
Working with extremal equations \eqref{eq:exteqs} and considering
\[
u:=\frac{u}{\|u\|_{L^\infty(Q)}+\|f\|_{L^p(Q)}},
\]
we can assume that $\|u\|_{L^\infty(Q)}\leq 1$ and $\|f\|_{L^p(Q)}\leq 1$.

Fix $r\in (0,1)$. 
Setting $M_r :=\sup_{Q_r}u$ and $m_r :=\inf_{Q_r}u$, we define 
\[
\o (r):=M_r-m_r\quad \mbox{for }r\in (0,1).\]

It is easy to observe that for $(x,t)\in Q$, 
$v(x,t):=M_r-u(rx,10+r^2(t-10))$ and 
$w(x,t):=u(rx,10+r^2(t-10))-m_r$ are nonnegative, 
$L^p$-viscosity supersolutions of \eqref{P+}. 
Hence, in view of Theorem \ref{thm:weak}, we find  constants $\e_0, C_0>0$ such that 
\[
\le(\int_{J_1} U^{\e_0} \,dxdt\ri)^{\fr{1}{\e_0}}
\leq C_0\le(
\low{J_2}\inf \, U+r^{2-\fr{n+2}{p}}\| f\|_{L^p(Q)}\ri)
\quad\mbox{for }U=v\,\,\mbox{and } U=w,\]
where $\e_0, C_0>0$ are the constants from Theorem \ref{thm:weak}. 
%Since we  may assume that $C_0>1$,  
Setting $a_0=2-\fr{n+2}{p}$ and $C_1=2^{\max\{ 0,\fr{1}{\e_0}-1\}}C_0|J_1|^{-\fr{1}{\e_0}}$, we have
\[
\begin{array}{rl}
\o(r)=\le(\dis \fr{1}{|J_1|}\int_{J_1}\o(r)^{\e_0} \,dxdt\ri)^{\fr{1}{\e_0}}\leq&
 C_1\le(\low{J_2}\inf \, v+\low{J_2}\inf \, w +r^{\a_0} \ri)\\
\leq& C_1\le( \o(r)-\low{J_2^r}\sup \, u+\low{J_2^r}\inf \, u +r^{\a_0} \ri),
\end{array}
\]
where $J_2^r:=(-r,r)^n\ti (10-r^2,10]$. 
Since we may suppose $C_1>1$, noting $Q_{\fr{r}{10}}\subset J_2^r$, we have 
\[
\o\le(10^{-1}r\ri)\leq \low{J^r_2}\sup \, u-\low{J^r_2}\inf \, u\leq 
\g\o(r) +r^{\a_0},\]
where $\g=\fr{C_1-1}{C_1}$. 
Therefore, in view of the standard argument 
(e.g. Lemma 8.23 in \cite{GilTru83}), 
setting $\a=\min\{ -\fr{\log\g}{\log 10},\a_0\}\in (0,1)$, we conclude the proof. 
\end{proof}

%%%%%%%%%%%%%%%%
%%%%%%%%%%%%%%%   4. 2.  Harnack inequality
%%%%%%%%%%%%%%%%

\subsection{Harnack inequality}

In order to prove the Harnack inequality we need the local maximum principle for $L^p$-viscosity subsolutions of 
\begin{equation}\label{4-6}
u_t+\P^-(D^2u)-\mu|Du|-f=0 \quad\mbox{in } Q.
\end{equation}
Following the arguments of \cite{CafCab}, 
we show that the weak Harnack inequality implies the local maximum principle. 
We note that to show Proposition \ref{prop:local}, 
we can apply the arguments of \cite{ImbSil13}, which 
is based on the standard one (e.g. \cite{GilTru83}). 
% with a clever modification. 

In this paper, we present a parabolic version of the method
of \cite{CafCab} (see also \cite{KS5}). 
We first show a blow-up lemma.

%%%%%%%%%%%%%%%    Lemma   4.  3

\begin{lem}\label{L4LMP1}
Let \eqref{Apq1} hold and let $f\in L^p_+(Q)$ and $\mu\in L^q_+(Q)$. 
Assume  that 
\[\| f\|_{L^p(\O)}\leq\e_1,\]
where $\e_1>0$ is the constant in the proof of Theorem \ref{thm:weak}. Suppose
that $v\in C(Q)$ is an $L^p$-viscosity subsolution of 
\eqref{4-6} satisfying 
\begin{equation}\label{LLMP}
|\{ (x,t)\in J_1 \ | \ v(x,t)\geq  s\}|\leq A_0s^{-\b_0}\quad (\forall s\geq 1)
\end{equation}
where $\b_0>0$ and $A_0>1$. % are in Remark \ref{RwHI}. 
Then, there exist $\nu =\nu (n,\L,\l,p,q,\b_0,A_0)>1$, $n_0=n_0(n,\L,\l,p,q,\b_0,A_0)\in\N$ and $\ell_j=\ell_j(n,\L,\l,p,q,\b_0,A_0)\in (0,1)$ for $j\in \N$ such that 
$\sum_{j=1}^\infty \ell_j<\infty$, and if $v$ satisfies
$v(x_0,t_0)\geq \nu^{j-1}$ for some $j\geq n_0$ and $(x_0,t_0)\in \ol J_3$, 
then it follows that 
\[
\low{\hat Q_j}\sup \, v\geq \nu^j,\]
where $J_3=(-\fr12,\fr12)^n\ti (\fr14,\fr12]$ and 
$\hat Q_j=(x_0,t_0)+(-\ell_j,\ell_j)^n\ti (-\fr{\ell_j^2}{10},0]$ (See Figure 4). 
\end{lem}

\begin{rem}
The constants $A_0$ and $\b_0$ in Lemma \ref{L4LMP1} will be those in Remark \ref{RwHI}. 
\end{rem}

%%%%%%%%%%%%%%%
%%%%%%%%%%%%%%%     Fig  4
%%%%%%%%%%%%%%%

\begin{picture}(300,150)
\put(50,10){\vector(1,0){300}}%\UTF{00C2}\UTF{0089}\UTF{00C2}\UTF{00A1}\UTF{00C2}\UTF{008E}\UTF{00C2}\UTF{00B2}%
\put(200,0){\vector(0,1){150}}%\UTF{00C2}\UTF{008F}c\UTF{00C2}\UTF{008E}\UTF{00C2}\UTF{00B2}%
\put(100,10){\framebox(200,100)}%Q
\put(100,10){\framebox(200,50)}
\put(150,35){\framebox(100,25)}% L
\put(145,30){\framebox(10,5)}% L1
%\put(153,57){$\cdot$}
\put(120,25){$\hat Q_j$}
\put(132,28){$\rightarrow$}
\put(105,50){$(x_0,t_0)$}
\put(139,38){$\searrow$}
%\put(173,64){\framebox(54,72)}% L2
%\put(119,136){\line(1,0){162}}% L3
%\put(119,136){\line(0,1){140}}
%\put(50,10){\line(0,1){200}}
\put(170,-2){$(0,0)$}
\put(350,0){$x$}
\put(300,0){$1$}
%\put(248,255){$3$}
\put(190,140){$t$}
\put(190,115){$1$}
%\put(
%\put(155,45){$\bullet  (x_0,t_0)$}
\put(220,50){$J_3$}
\put(202,67){$\fr12$}
\put(202,22){$\fr14$}
\put(300,60){$\Bigl\}K_1$} 
\put(270,30){$J_1$}
%\put(205,53){$\leftarrow N_1$}
%\put(210,100){$N_2$}
%\put(230,200){$N_3$}
\put(130,-20){Figure $4$. The cubes $J_3$, $\hat Q_j$.}
\end{picture}

\vspace{1cm}

%%%%%%%%%%%   Proof of Lemma 4. 3. 

\begin{proof} 

We first fix $\nu:=\fr{\a}{\a-1}>1$ (i.e. $\a=\fr{\nu}{\nu-1}$), where 
\[
\a:=2(2A_0)^{\fr{1}{\b_0}}>1.\]

Assume $\sup_{\hat Q_j}v\leq \nu^j$.  
We will arrive at a contradiction provided 
\begin{equation}\label{LMPell}
\ell_j:=2(2^{\b_0+1}A_0\nu^{-j\b_0})^{\fr{1}{n+2}}.
\end{equation}

Choose $j_0\in\N$ such that $\ell_j\leq \fr{1}{2\sqrt{10}}$ 
for $j\geq j_0$. 
For $j\geq j_0$, setting 
\[
w^j(x,t)=\fr{\nu}{\nu-1}\le\{ 1-\nu^{-j}v\le(x_0+
\ell_j
x,t_0+\ell_j^2(t-10) \ri)\ri\}\geq 0\quad 
\mbox{in } Q,
\]
we note that 
\[
\low{J_2}\inf \, w^j\leq w^j(0,10)=\fr{\nu}{\nu-1}\le\{ 1-
\nu^{-j}v(x_0,t_0)\ri\}\leq 1\]
and that $w^j$ is an $L^p$-viscosity supersolution of
\[
w^j_t+\P^+(D^2w^j)+\hat \mu_j |Dw^j|-\hat f_j=0\quad\mbox{in }Q,\]
where $\hat\mu_j (x,t)=\ell_j\mu (x_0+\ell_jx,
t_0+\ell_j^2(t-10)%\fr{\ell_j^2}{10}
)$ and 
$\hat f_j (x,t)=\fr{\ell_j^2}{(\nu-1)\nu^{j-1}}f(x_0+\ell_jx,t_0+\ell_j^2(t-10)%\fr{\ell_j^2}{10}
)$. 
%where $\nu=\fr{2(4A)^{\fr{1}{\e}}}{2(4A)^{\fr{1}{\e}}-1}$. 
Since it follows that 
\[
\|\hat\mu_j\|_{L^q(Q)}=\ell_j^{1-\fr{n+2}{q}}\|\mu\|_{L^q(\hat Q_{j})},\quad\mbox{and}\quad 
\| \hat f_j\|_{L^p(Q)}=\fr{\nu^{1-j}}{\nu-1}\ell_j^{2-\fr{n+2}{p}}\| f\|_{L^p(\hat Q_{j})},\]
there exists  an integer $n_0=n_0(n,\l,\L, p,q,\b_0,A_0)\geq j_0$ such that $\|\hat \mu_j\|_{L^q(Q)}\leq \|\mu\|_{L^q(Q)}$ and $\|\hat f_j\|_{L^p(Q)}\leq%\| f\|_{L^p(Q)}\leq
 \e_1$ for $j\geq n_0$. 

In view of Remark \ref{RwHI}, we thus have 
\[
\le|\le\{ (x,t)\in J_1 \ \le| \ w^j(x,t)\geq \fr12 \a\ri\} \ri.\ri|
\leq A_0\le(\fr{2}{\a}\ri)^{\b_0}=\fr12.
\]
Hence, we have 
\[
\le|\le\{ (y,s) \in \hat C_j \ \le| \ v(y,s)\leq \fr12 \nu^j
\ri.\ri\}\ri|\leq \fr12 \ell_j^{n+2}
,\]
where 
\[\hat C_j=(x_0,t_0)+\le(-\ell_j,\ell_j\ri)^n\ti 
\le(-10\ell_j^2,-\fr{19}{2}\ell_j^2\ri).\]

On the other hand, since $\hat C_j\subset J_1$, by \eqref{LLMP}, we have 
\[
\le|\le\{(y,s) \in \hat C_j \ \le| \ v(y,s)\geq \fr12 \nu^j\ri.\ri\}\ri|\leq A_0\le(2\nu^{-j}\ri)^{\b_0}.\]
Thus, noting 
\[
|\hat C_j|=
2^{n-1}\ell_j^{n+2}
\leq \fr12\ell_j^{n+2}+A_0\le(
2\nu^{-j}\ri)^{\b_0}
,
\]
we have 
\[
\ell_j\leq (2^{\b_0+1}A_0\nu^{-j\b_0})^{\fr{1}{n+2}},
\]
which contradicts \eqref{LMPell}. 
\end{proof}

We can now show the local maximum principle for $L^p$-viscosity subsolutions. 

%%%%%%%%%%%   Proposition  4.  5.  Local maximum principle 

\begin{prop}%Local Maximum Principle
\label{prop:local}
Let $(\ref{Apq1})$ hold and let $f\in L^p_+(Q)$ and $\mu\in L^q_+(Q)$. 
Then, for any $\e_0\in (0,\b_0)$, 
there exists a constant $C_3=C_3(n,\L,\l, p, q, \|\mu\|_{L^q(Q)},\e_0)>0$ 
such that any $L^p$-viscosity subsolution $u\in C(Q)$ of \eqref{4-6}  
satisfies 
\[%\label{equ:2}
\low{J_3}{\sup}\, u\leq C_3\le(
\|u^+\|_{L^{\e_0}(J_1)} +\|f\|_{L^p(Q)}\ri),
\]
where $\b_0>0$ is the constant in Remark \ref{RwHI}. 
%where $J_3:=(-\fr12,\fr12)^n\ti (\fr14,\fr12]$, and $J_1=(-1, 1)^n\ti(0, \fr{1}{2}]$  as in section 3. 
\end{prop}

\begin{proof}
Choose $(y_0,s_0)\in \ol J_3$ such that
\[
\low{J_3}\sup \, u=u(y_0,s_0).\]

%For $r\in (0,\fr{1}{\sqrt{20}})$, 
Setting 
\[N_0=\le(A_0^{-1}\int_{J_1}  (u^+)^{\e_0}\,dxdt\ri)^{\fr{1}{\e_0}}%(y+r\cdot,s-10r^2+r^2\cdot )\|_{L^\e (J_1)
+2\e_1^{-1}\|  f\|_{L^p(Q)},\]
where $\e_1>0$ is from the proof of Theorem \ref{thm:weak}, 
we observe that 
$v:=N_0^{-1}u$ is an $L^p$-viscosity subsolution of 
\[
v_t+\P^-(D^2v)-\mu |Dv|-\fr{1}{N_0} f=0\quad \mbox{in }Q.\]

We note that for $s\geq 1$, we have
\[
|\{ (x,t)\in J_1 \ | \ v(x,t)\geq s\} |\leq \dis\fr{1}{s^{\e_0}}\int_{J_1}v^{\e_0} \,dxdt
\leq A_0 s^{-\e_0}.
\]

Let  $\nu>1$, $n_0\in\N$ and $\ell_j>0$ be the constants in Lemma \ref{L4LMP1} when $\b_0=\e_0$. 
There exists $n_1\geq n_0$ such that 
\[
\sum_{j=n_1}^\infty\ell_j\leq \fr14.\]
Now, suppose that there is $(y_0,s_0)\in \ol J_3$ such that 
\[
v(y_0,s_0)\geq \nu^{n_1-1}.\]
In view of Lemma \ref{L4LMP1}, for $j\in\N$, 
we can find $(y_j,s_j)\in (y_{j-1},s_{j-1})+[-\ell_{j+n_1-1},\ell_{j+n_1-1}]^n 
\ti [-\fr{\ell^2_{j+n_1-1}}{10},0]$ 
%Q_{\ell_{n_1+j-1}}(y_{j-1},s_{j-1})$
 such that 
\[
v(y_j,s_j)\geq \nu^{n_1+j-1}.\]
Because $(y_j,s_j)\in [-\fr34,\fr34]^n\ti [\fr18,\fr12]$, this contradicts 
that $v\in C(Q)$. 
Therefore, we conclude the proof. 
\end{proof}

%%%%%%%%%%%%%%%%%\UTF{00C2}\UTF{0081}@\UTF{00C2}\UTF{0081}@Harnack Inequality
Using the weak Harnack inequality, together with Proposition 
\ref{prop:local}, we can obtain the Harnack inequality which we state without proof.

\begin{cor}%Harnack Inequality
\label{thm:harnack}
Let $(\ref{Apq1})$ hold and let $f\in L^p(Q)$ and $\mu\in L^q(Q)$.
There is a constant $C_4=C_4(n, \L,\l,  p, q,\|\mu\|_{L^q(Q)})>0$ such that any nonnegative $L^p$-viscosity solution $u\in C(\ol{Q})$ of $(\ref{4-1})$ satisfies
\[
%\label{eq:harnack}
\low{J_3}{\sup}\, u\leq C_4\le(\low{J_2}{\inf}\,u +\|f\|_{L^p(Q)}\ri), 
\]
where $J_3=(-\fr12,\fr12)^n\ti (\fr14,\fr12]$ and 
$J_2=(-1, 1)^n\ti(9, 10]$. 
\end{cor}

%%%%%%%%%%%%%%%%
%%%%%%%%%%%%%%%%         Section  5
%%%%%%%%%%%%%%%%

\section{Remarks on the superlinear growth case}

In this section, we exhibit several properties of $L^p$-viscosity solutions of \eqref{4-1}, where $G$ satisfies \eqref{UP}, \eqref{4-5} and, in place of \eqref{4-4}, 
\begin{equation}\label{EqSuper}
\le\{\begin{array}{c}
\mbox{there are }m>1\mbox{ and }\mu\in L_+^q(Q)\mbox{ for }q>n+2\mbox{ such that}\\
|G(x,t,\xi,O)|\leq \mu (x,t)|\xi|^m\mbox{ for }(x,t)\in Q\mbox{ and }
\xi\in \R^n.
\end{array}\ri.
\end{equation}
More precisely, we present a remark on the ABP maximum principle in 
\cite{KS2}, and an existence result corresponding to that in 
\cite{KS4}, with which we show the weak Harnack inequality for 
$L^p$-viscosity supersolutions of \eqref{4-1} under \eqref{EqSuper}.  

If \eqref{UP}, \eqref{4-5} and \eqref{EqSuper} are satisfied
then if $u\in C(Q)$ is an 
$L^p$-viscosity subsolution (resp., supersolution) of \eqref{4-1}, then 
it is an $L^p$-viscosity subsolution (resp., supersolution) of 
\[
u_t+\P^-(D^2u)-\mu |Du|^m-f=0\quad
\le(\mbox{resp., } \ u_t+\P^+(D^2u)+\mu |Du|^m-f=0\ri)\quad\mbox{in }Q.\]
To establish the ABP maximum principle and the weak Harnack inequality we only need to work with the above extremal inequalities.

%\begin{rem}\label{referee2}

%\end{rem}

%%%%%%%%%%%%%%%%%%
%%%%%%%%%%%%%%%%%%     5. 1.  A remark
%%%%%%%%%%%%%%%%%%

\subsection{A remark on the ABP maximum principle}

In this section, to comply with the setup of \cite{KS2}, $Q=\O\times(0,T]$, where $0<T\leq 1$ and the domain $\O$ satisfies
\begin{equation}\label{eq:Omega1}
\O\subset \{ x\in\R^n \ | \ |x|<1\}.
\end{equation}
We recall the ABP maximum principle from \cite{KS2}. The estimates there seem 
a little complicated. 
However, if we carefully examine them, we can give simple statements as below. 

%%%%%%%%%%%%  Proposition  5. 1.  ABP for superlinear

\begin{prop}\label{propSmall}(Theorems 3.11 and 3.12 of \cite{KS2})
Let \eqref{Apq1} hold with $q<+\infty$. Let \eqref{UP}, \eqref{4-5} and \eqref{EqSuper} be satisfied and let
\begin{equation}\label{Apqm1}
p>\fr{(m-1)q(n+2)}{mq-n-2}.
\end{equation}
There exist $\d =\d (n,\L,\l,m,p,q)>0$  and $C=C(n,\L,\l, m,p,q)>0$ such that 
if $u\in C(\ol Q)$ is an $L^p$-viscosity subsolution of 
\eqref{4-1}, 
and 
\begin{equation}\label{Smallness}
\| f\|_{L^p(Q)}^{m-1}\| \mu\|_{L^q(Q)}\leq\d,
\end{equation}
then 
\[
\under Q\sup \ u\leq \under{\p_pQ}\sup \ u+
   C\| f\|_{L^p(Q)}.
\]
\end{prop}

%%%%%%%%%%%%   Remark  5. 2.  

\begin{rem} We note that \eqref{Apqm1} is satisfied when $n+2\leq p\leq q, q>n+2$, 
and \eqref{Apqm1} is equivalent to
\[
mq(n+2-p)<(n+2)(q-p),\]
which is (iv) of \eqref{pqm}.
We also remark that when $q=+\infty$, the ABP maximum principle does not require any smallness condition and 
can be found in Theorems 3.7 and 3.8 of \cite{KS2}. Condition \eqref{Apqm1} then reduces to $p>(m-1)(n+2)/m$,
which is the inequality in (i) of \eqref{pqm}.
\end{rem}

We show here that the smallness condition \eqref{Smallness} can be removed, however the estimate becomes more complicated. 

%%%%%%%%%%  Theorem  5. 3.  ABP  without  Smallness

\begin{thm}
\label{z5} 
Let \eqref{Apq1} hold with $q<+\infty$. Let \eqref{UP}, \eqref{4-5}, \eqref{EqSuper} and \eqref{Apqm1} be satisfied.
There exists $C=C(n,\L,\l, m,p,q )>0$ such that 
if $u\in C(\ol Q)$ is an $L^p$-viscosity subsolution of 
\eqref{4-1}, 
then 
$$
\under Q\sup \ u\leq \under{\p_pQ}\sup \ u+C\left(1+\| f\|_{L^p(Q)}^{(m-1)q}\| \mu\|_{L^q(Q)}^q\right)^{\frac{p-1}{p}}\| f\|_{L^p(Q)}.$$
\end{thm}

\begin{proof}
By considering $u:=u-\sup_{\p_pQ} u$, we may assume that $\sup_{\p_pQ} u\leq 0$. When \eqref{Smallness} does not hold,  it is easy to see that we can find a partition
$0=t_0<t_1<\cdots<t_k=T$ such that, setting  $Q_i:=\Omega\times[t_{i-1},t_i]$, $i=1,\cdots,k$, and 
$\hat \d:=\| f\|_{L^p(Q)}^{1-m}\d$, 
we have 
\[
\|\mu\|_{L^q(Q_i)}\leq\hat\d\quad\mbox{for } i=1,\cdots ,k, \quad\mbox{where}\,\, k
\leq 1+\delta^{-q}\| f\|_{L^p(Q)}^{(m-1)q}\| \mu\|_{L^q(Q)}^q.
\]
%Notice that we can choose $k=\le[\|\mu\|_{L^q(Q)}\| f\|_{L^p(Q)}^{m-1}\ri]+1$. 
%This can be done for instance if $$k=\left[\frac{\|\mu\|_p\|f\|_p^{m-1}}{\delta}\right]+1.$$
By Proposition \ref{propSmall}, we then have
\[
\under{Q_i}\sup \ u\leq \under{\p_pQ_i}\sup \ u+C\|f\|_{L^p(Q_i)}\quad\mbox{for } i=1,\cdots,k.
\]
Let $(\hat x,\hat t)\in \ol Q_i$ satisfy $\sup_Q u=
u(\hat x,\hat t)$ for some $i\in \{1,\ldots,k\}$. 
Then
\[
\under{Q}\sup \ u\leq \under{\p_pQ_i}\sup \ u+C\| f\|_{L^p(Q_i)}\leq \max(0,\under{Q_{i-1}}\sup \ u)+ C\|f\|_{L^p(Q_i)}.
\]
But
\[
\under{Q_{i-1}}\sup \ u\leq \under{\p_pQ_{i-1}}\sup 
u+C\| f\|_{L^p(Q_{i-1})}\leq \max(0,\under{Q_{i-2}}\sup \ u)+ C\|f\|_{L^p(Q_{i-1})}.\]
Therefore, continuing this procedure, we obtain
\[
\under{Q}\sup \ u\leq C\sum_{i=1}^k \| f\|_{L^p(Q_i)}.
\]
Now
\[
\sum_{i=1}^k \| f\|_{L^p(Q_i)}\leq k^{\frac{p-1}{p}}\| f\|_{L^p(Q)}\leq \left(1+\delta^{-q}\| f\|_{L^p(Q)}^{(m-1)q}\| \mu\|_{L^q(Q)}^q\right)^{\frac{p-1}{p}}\| f\|_{L^p(Q)}.
\]
\end{proof}

\subsection{Existence of strong solutions}

In this subsection, for the sake of simplicity, $\O$ is as in \eqref{eq:Omega1}
and we assume that $\p \O$ is $C^{1,1}$. 
We discuss  the existence of $L^p$-viscosity solutions of parabolic extremal PDE, 
\begin{equation}\label{5Extremal}
u_t+\P^\pm (D^2u)\pm \mu |Du|^m
=f\quad\mbox{in }Q:=\O\ti (0,1],
\end{equation}
where $m>1$, $f\in L^p(Q)$ and $\mu\in L^q(Q)$. 
%Since we use a result from \cite{DKX}, we will suppose $p>n+1$. 
%However, our results below presents an existence of strong solutions while in Proposition 2.4, we only show the existence of %strong sub- or supersolutions. 
Since we do not know a precise proof of $W^{2,1}_p$-estimates 
near $\p_pQ$ of \cite{W}, possibly for $p\leq n+1$, 
(though it was mentioned in \cite{W} without a proof), we will use 
global estimates for $p>n+1$ from \cite{DKX} to show a different type of estimates. 
% based on Aubin-Lions' imbeddings.XXXX CHECK! 
Thus we will assume that $p>n+1$. We first recall a global estimate for $L^p$-strong solutions 
of extremal PDE with no first derivative terms. 

%%%%%%%%%%%%  Proposition  5. 4.  Existence when  \mu = 0

\begin{prop}\label{5GlobalEx}(e.g. Theorem 1.1 of \cite{DKX})
Let $\p\O$ be $C^{1,1}$ and  $p>n+1$. 
Then, there exists a constant $C=C(n,p,\L,\l,{\rm diam}(\Omega),\partial\Omega)>0$ such that for every $f\in L^p(Q)$ and 
$\psi\in W^{2,1}_p(Q)$, there exists a unique $u\in C(\ol Q)\cap
W^{2,1}_p(Q)$ such that 
\[\le\{\begin{array}{rl}
u_t+\P^+(D^2u)=f&\mbox{a.e. in }Q,\\
u=\psi&\mbox{on }\p_pQ,
\end{array}\ri.
\]
and
\[
\| u\|_{W^{2,1}_p (Q)}\leq C\le(\| f\|_{L^p(Q)}+\|\psi\|_{W^{2,1}_p(Q)}\ri)
.\]
\end{prop}

For the elliptic case, in \cite{KS4}, the existence of $L^p$-strong solutions of 
extremal PDE with superlinear growth in the first derivatives was obtained
assuming that $\|\mu\|_{L^q(Q)}$ is small enough. 
Following the idea of \cite{KS4}, we establish the corresponding 
existence result for $L^p$-strong solutions of \eqref{5Extremal}. 
%Although we need some smallness hypothesis as in the elliptic case, in applications (e.g. the weak Harnack inequality below), we will  avoid this restriction. 
%%%%%%%%%%%%%%%%    Theorem  5. 5.   Existence of Strong solutions 

\begin{thm}\label{thm:ExStrong}
Let $\p\O$ be $C^{1,1}$, $n+1<p\leq q$, $q>n+2$, $f\in L^p(Q)$, $\mu\in L^q(Q)$ and $\psi\in W^{2,1}_p(Q)$. 
Assume that one of the following conditions holds:
\begin{equation}\label{pqm}
\le\{\begin{array}{rl}
(i)&n+1<p<n+2, m(n+2-p)< n+2, q=\infty,\\
(ii)&p\geq n+2, q=\infty,\\
(iii)&n+2<p= q<\infty,\\
(iv)&n+1<p<q, q>n+2, mq(n+2-p)< (n+2)(q-p).
\end{array}
\ri.
\end{equation}
Assume also that
\begin{equation}\label{pqmr}
\le\{
\begin{array}{ll}
r=pm&\mbox{for }(i), (ii)\\
r=\infty&\mbox{for }(iii),\\
r=\fr{mpq}{q-p}&\mbox{for }(iv).
\end{array}
\ri.
\end{equation}
Then, there exists $\d_1=\d_1(n,\L,\l,p,q,m)>0$ such that 
if  
\begin{equation}\label{ExSmall}
\| \mu\|_{L^q(Q)}\le(\| f\|_{L^p(Q)}+\|\psi\|_{W^{2,1}_p(Q)}\ri)^{m-1}
\leq\d_1,
\end{equation}
then there exist 
$L^p$-strong solutions $u\in W^{2,1}_p(Q)$ of
\begin{eqnarray*}%\label{5Dirichlet}
\le\{
\begin{array}{rl}
u_t+\P^\pm (D^2u)\pm\mu |Du|^m=f&a.e.\mbox{ in }Q,\\
u=\psi&\mbox{on }\p_pQ.
\end{array}\ri.
\end{eqnarray*}
Moreover, there exists  $\hat C=\hat C(n,\L,\l,p,q,m,{\rm diam}(\O),\partial\O)>0$ such that
\begin{equation}\label{eq:w2psuperlin}
\| u\|_{W^{2,1}_p(Q)}\leq \hat C\le(\| f\|_{L^p(Q)}+\|\psi\|_{W^{2,1}_p(Q)}\ri).
\end{equation}
\end{thm}

%%%%%%%%%%%%   Remark  5. 6. 

\begin{rem}
We note that  in $(iv)$ of \eqref{pqm}, if $p\geq n+2$, then the third inequality 
 automatically holds. 
\end{rem}

%Moreover, for $T>0$ and a Banach space $X$,  we denote by 
%$$C([0,T];X)=\le\{ v:[0,T]\to X: \mbox{ continuous } \le| \ \under{0\leq t\leq T}\sup \| v(t)\|_X<\infty\ri.\ri\},$$ 
%$$
%L^p(0,T;X)=\le\{ u:(0,T)\to X \ \le| \ \int_0^T\| u(t)\|_X^p dt <\infty\ri.\ri\},$$
%where $\|\cdot\|_X$ denotes the norm of $X$. 
%$L^p(0,T;X)$ the completion of continuous functions $u:(0,T)\to X$ with the norm defined by
%$$\| u \|:=\left(\int_0^T \| u(t)\|_X^pdt\right)^{\frac 1p}.$$

\begin{proof}
We will do the proof only for the case of $\P^+$. For $r$ in \eqref{pqmr}, we define a mapping $K:W^{1,0}_r(Q) \to W^{2,1}_p(Q)$ in the following way.
For $v\in W^{1,0}_r(Q)$, in view of Proposition \ref{5GlobalEx}, 
we find  a unique solution $u:=Kv\in W^{2,1}_p(Q)$ of 
\[
\le\{\begin{array}{rl}
u_t+\P^+(D^2u)+ \mu|Dv|^m=f&\mbox{a.e. in }Q,\\
u=\psi&\mbox{on }\p_pQ.
\end{array}\ri.\]
Since $\| Kv\|_{W^{2,1}_p(Q)}\leq C(\| f\|_{L^p(Q)}+\| \mu|Dv|^m\|_{L^p(Q)}+\| \psi\|_{W^{2,1}_p(Q)})$ holds for some $C>0$, 
noting 
\[
\|\mu|Dv|^m\|_{L^p(Q)}\leq C\|\mu \|_{L^q(Q)}\| Dv\|_{L^r(Q)}^m,
\]
we can argue like in the proof of Theorem 3.1 of \cite{KS4} to find a sufficiently large $\alpha$ and small $\d_1>0$ such that if
$R=\alpha(\| f\|_{L^p(Q)}+\|\psi\|_{W^{2,1}_p(Q)})$, then $K:\B_R\to W^{2,1}_p(Q)\cap \B_R$ is a continuous map when \eqref{ExSmall} holds, where 
\[
\B_R=\le\{ v\in W^{1,0}_r(Q) \ \le| \ \| v\|_{W^{1,0}_r(Q)}\leq R \ri.\ri\}.\]
Since $W^{2,1}_p(Q)$ is compactly imbedded in $W^{1,0}_r(Q)$ (see the next proposition), we conclude the proof by the Schauder fixed point theorem as in \cite{KS4}. 
\end{proof}

%Thanks to Aubin-Lions lemma (see Corollary 8 in \cite{S} for instance), 
%it is known that under \eqref{pqm} and \eqref{pqmr}, we have 
%$$W^{2,1}_p(Q) \Subset W^{1,0}_r(Q)$$
%because $W^{2}_p(\O)\Subset W^{1}_r(\O)\subset L^p(\O)$, 
%where $\Subset $ means the compact imbedding. 

\def\A{{\cal A}}
%%%%%%%%%%%%%%%%%%
%%%%%%%%%%%%%%%%%%11/12TATEYAMA WROTE
%%%%%%%%%%%%%%%%%%

For the reader's convenience, we provide a proof of compact imbeddings of parabolic Sobolev spaces. More general results
for compact imbeddings of anisotropic Sobolev spaces can be found in \cite{BIN1} and \cite{BIN2} (see in particular
Theorem 10.2 of \cite{BIN1} and Theorem 26.3.5 of  \cite{BIN2}).
 %We refer to \cite{LSU}, \cite{BIN1} and \cite{BIN2}.
%%%%%%%%%%%%%%%%%
%%%%%%%%%%%%%%%%%
%%%%%%%%%%%%%%%%%  Compact imbedding theorem 
\begin{prop}%(cf. Theoerm 26.3.5 of \cite{BIN2})
Let $\p\O$ be Lipschitz.  
Assume that $1\leq p\leq r$ satisfy one of the following conditions: 
\begin{equation}\label{prr}
\le\{\begin{array}{rl}
(i)&p<n+2, p\leq r< \fr{p(n+2)}{n+2-p},\\
(ii)&p= n+2\leq r<\infty,\\
(iii)& n+2<p<\infty, r=\infty.
\end{array}
\ri.
\end{equation}%\end{align}
Then, $W^{2, 1}_p(Q)$ is compactly imbedded in 
$W^{1,0}_r(Q)$. 
\end{prop}

\begin{proof}
%We first recall the anisotropic imbedding (e.g.  Theorem ?? in \cite{LSU}); there exists $C>0$ such that
%$$
%\| u\|_{W^{1,0}_r(Q)}\leq C\| u\|_{W^{2,1}_p(Q)}\quad \mbox{for }u\in W^{2,1}_p(Q).
%$$
Under assumption \eqref{prr}, by Lemma 3.3 of \cite{LSU}, 
 it follows that there exist $\e'>0$ and $C>0$ such that for any $\e\in(0, \e')$, we have for $u\in W^{2,1}_p(Q)$
\begin{equation}\label{imbedd}
\|u\|_{L^r(Q)}+\|Du\|_{L^r(Q)}\leq C \e^{\alpha}\le(\|D^2u\|_{L^p(Q)}+\|u_t\|_{L^p(Q)}\ri)+C\e^{-(2-\alpha)}\|u\|_{L^p(Q)},\end{equation}
where $\alpha=1-\fr{n+2}{p} +\fr{n+2}{r}>0$ for $r<\infty$, or $\alpha=1-\fr{n+2}{p}>0$ for $r=\infty$. Here, $C$ is independent of $u$ and $\e\in(0, \e')$. (A better inequality is true for
$\|u\|_{L^r(Q)}$ but we do not need it here.)
%Hence, this inequality yields $W^{2, 1}_p(Q)\subset W^{1,0}_r(Q)$ and (\ref{subset}).
	
%Let $\A$ be a bounded subset of $W^{2, 1}_p(Q)$. 
In view of \eqref{imbedd}, it is thus enough to show that a bounded subset of $W^{2, 1}_p(Q)$ is  compact in $L^p(Q)$. 
However this is clear since $W^{2, 1}_p(Q)\subset W^{1}_p(Q)$ (when we consider $Q$ as a subset of $\R^{n+1}$) and the mapping $I:W^{1}_p(Q)\to L^p(Q)$ is compact by the standard compact Sobolev imbedding theorem (see e.g. Theorem 7.26
of \cite{GilTru83}). 
\end{proof}

%%%%%%%%%%%%%%   Remark 5. 8

\begin{rem}
We remark that for case \eqref{prr}-(iii) a stronger result is true, namely that $W^{2, 1}_p(Q)$ is compactly imbedded in the parabolic space 
$C^{1+\alpha}(Q)$ for $\alpha=1-\fr{n+2}{p}$.
\end{rem}

%%%%%%%%%%%%%%%%%%%
%%%%%%%%%%%%%%%%%%%  5. 3.  Weak Harnack
%%%%%%%%%%%%%%%%%%%

\subsection{Weak Harnack inequality}

Using Theorem \ref{thm:ExStrong},  we establish 
the weak Harnack inequality for $L^p$-viscosity supersolutions of uniformly parabolic PDE with 
superlinear growth in $Du$.  
We refer to \cite{KS4} for an analogous elliptic result.

In this subsection, we again set $Q:=(-10,10)^n\ti (0,10]$. 
In what follows, we will utilize the same notation as that in Figure 1.
We will construct a barrier function for \eqref{5Extremal} when $m>1$. 
This will require a slightly more careful analysis than that in the elliptic case.

%%%%%%%%%%%    Lemma  5. 9.   barrier  

\begin{lem}\label{lem:m_barrier}
Assume that \eqref{Apq1} %or \eqref{Apq2} 
holds. 
Then, there exists $\d_2=\d_2(n,\L, \l,q,m)>0$ such that 
if $\mu\in L^q(Q)$ satisfies
\[
\|\mu\|_{L^q(Q)}\leq \d_2,
\]
then there exist $\phi\in W^{2,1}_q(Q)\cap C(\ol Q)$ and $g\in L^q(Q)$ 
such that 
\begin{equation}\label{eq:strong-g}
\le\{\begin{array}{rcll}
\phi_t+\P^+(D^2\phi)+\mu |D\phi|^m&\leq& g&\mbox{a.e. in }Q,\\
\phi&\geq& 2&\mbox{in }K_2,\\
\phi&=&0&\mbox{in }\p_pQ,\\
\mathrm{supp }\ g&\subset&K_1.&\\
\end{array}
\ri.
\end{equation}
\end{lem}

\begin{proof}
We first introduce a smooth, nonnegative $\eta:\ol Q\to [0,1]$ satisfying 
\[%\label{5eta}
\le\{\begin{array}{ll}
(i)&\eta (x,t)=1\mbox{ for }(x,t)\in \ol Q\mbox{ if }|x|\geq 1\mbox{ or }t\geq 1,\\
(ii)&\eta (x,0)=0\mbox{ for }|x|\leq \fr12,\\
(iii)&\eta\in W^{2,1}_\infty(Q). 
\end{array}\ri. 
\]
We next choose a nonnegative function $\xi_0\in C^\infty (\R^n\ti [0,\infty))$ such that 
\[%\label{5xi0}
\le\{\begin{array}{cl}
(i)&\xi_0=0\mbox{ in } \R^n\ti [0,\infty)\setminus \{ (x,t)\in \R^n\ti [0,\fr14) \ | \ | x|<\fr12\},\\
(ii)&\xi_0(x,0)>0\mbox{ for }| x|<\fr12.
\end{array}\ri.
\]

As in the proof of Lemma 4.1 in \cite{KS4}, we claim that there exist $\d_2^0>0$ and $\s>0$ such that if 
$\mu'\in L^q(Q)$ satisfies $\|\mu'\|_{L^q(Q)}\leq \d^0_2$, then the strong solution 
  $\psi\in W^{2,1}_q(Q)$ of 
\begin{equation}\label{5barrierpsi}
\le\{\begin{array}{rl}
\psi_t+\P^+(D^2\psi )+\mu'|D\psi|^m=0&\mbox{a.e. in }Q,\\
\psi=\xi_0&\mbox{on }\p_pQ
\end{array}\ri.
\end{equation}
satisfies
\[
\psi\geq \s\quad\mbox{in }K_2.\]
Indeed, otherwise, there are nonnegative $\psi_k\in W^{2,1}_q(Q)\cap C(\ol Q)$ and $\mu_k\in L^q(Q)$ such that $\|\mu_k\|_{L^q(Q)}\leq\fr1k$ and $\psi_k$ is a strong solution of 
\eqref{5barrierpsi} with $\mu'$ replaced by $\mu_k$, such that $\inf_{K_2}\psi_k\leq \fr1k$, then (by \eqref{eq:w2psuperlin}) a subsequence  $\{ \psi_{k_j}\}_{j=1}^\infty $  converges 
uniformly in $\ol Q$ to some $\psi\in W^{2,1}_q(Q)$, and
 $\inf_{K_2} \psi=0$. Since $\psi$ is a strong solution of 
\[
\le\{\begin{array}{rl}
\psi_t+\P^+(D^2\psi)=0&\mbox{a.e. in }Q,\\
\psi=\xi_0&\mbox{on }\p_pQ,
\end{array}\ri.
\]
we can find $(\hat x,\hat t)\in K_2$ such that $\psi(\hat x,\hat t)=0$, 
which gives a contradiction as in the proof of Lemma \ref{lem:barrier}.

We now choose $\delta_2>0$ small enough so that 
\[
(4\sigma^{-1})^{m-1}\delta_2\leq \delta_2^0\quad\mbox{and}\quad (4\sigma^{-1})^{m-1}\delta_2\|\xi_0\|_{W^{2, 1}_p(Q)}^{m-1}\leq \delta_1,
\]
where $\d_1$ is from Theorem \ref{thm:ExStrong}.

In view of Theorem \ref{thm:ExStrong} and the above choice of $\delta_2$, if $\mu\in L^q(Q)$ satisfies $\|\mu\|_{L^q(Q)}\leq \delta_2$, then there exists $\psi^0\in C(\ol Q)\cap W^{2,1}_q(Q)$ 
such that 
\[
\le\{\begin{array}{rl}
\psi^0_t+\P^+(D^2\psi^0)+(4\sigma^{-1}\eta)^{m-1}\mu |D\psi^0|^m=0&\mbox{a.e. in }Q,\\
\psi^0=\xi_0 &\mbox{on }\p_p Q,
\end{array}\ri.\quad\mbox{and}\quad \psi^0\geq \sigma\quad\mbox{in } K_2.
\]
Setting $\psi=(2/\s)\psi^0$ and $\xi=(2/\sigma)\xi_0$, 
we observe that
\[
\le\{\begin{array}{rl}
\psi_t+\P^+(D^2\psi)+(2\eta)^{m-1}\mu |D\psi|^m=0&\mbox{a.e. in }Q,\\
\psi=\xi&\mbox{on }\p_pQ.
\end{array}
\ri.
\]
Furthermore, it is easy to check  that  $\phi:=\eta\psi$ 
satisfies 
\[
\phi_t+\P^+(D^2\phi)+\mu |D\phi|^m\leq g\quad\mbox{a..e. in 
 }Q,
\] 
where $g=\psi\eta_t+2^{m-1}\mu|\psi D\eta|^m+\P^+(
D\eta\otimes D\psi+D\psi \otimes 
D\eta+\psi D^2\eta)$, and $\phi$ and $g$ satisfy all the conditions required in \eqref{eq:strong-g}. 
\end{proof}

We will now show that the weak Harnack inequality holds under a smallness condition. 
Since we separate the weak Harnack inequality from the $L^\infty$-estimate, similarly to Theorem 4.2 in \cite{KS4},
we assume boundedness of 
supersolutions.

%%%%%%%%%%%%%    Theorem  5. 10.    Weak Harnack    m>1  with  Smallness

\begin{thm}\label{thm:m_WH}
Suppose that \eqref{pqm} holds and assume 
\[%\label{pqm_WH}
1<m<2-\fr{n+2}{q}. 
\]
Let $M\geq 0$, $f\in L^p_+(Q)$ and $\mu\in L^q(Q)$. 
Then, there exist $\d_3=\d_3(n,\l,\L,p,q,m,M)>0$, $C=C(n,\l,\L,p,q,m)>0$ and $\e_0=\e_0(n,\l,\L,p,q,m)>0$ such that if 
\[%\label{Small_m}
\|\mu\|_{L^q(Q)}\le(1+\|f\|_{L^p(Q)}^{m-1}\ri)<\d_3,
\]
and $u\in C(\ol Q)$ is an $L^p$-viscosity supersolution of
\[
u_t+\P^+(D^2u)+\mu |Du|^m+f=0\quad\mbox{in }Q
\]
satisfying $0\leq u\leq M$ in $Q$, then 
\[
\le(\int_{J_1}u^{\e_0} \,dxdt\ri)^{\fr{1}{\e_0}}\leq C\le(\under{J_2}\inf \ u+\|f\|_{L^p(Q)}\ri).\]
\end{thm}

\begin{proof}
The proof follows the arguments of the proof of Theorem 4.2 of \cite{KS4} so we just sketch it. We first reduce to the case of $f=0$. Let $\d_1$ be from Theorem \ref{thm:ExStrong} and let
\[
\|\mu\|_{L^q(Q)}(2\|f\|_{L^p(Q)})^{m-1}\leq \d_1.
\]
%We notice that if \eqref{pqm_WH} holds then \eqref{pqm} is satisfied. 
Let $w\in W^{2,1}_p(Q)$ be from Theorem \ref{thm:ExStrong} such that
\[
\le\{
\begin{array}{rl}
w_t+\P^- (D^2w)-2^{m-1} \mu|Dw|^m-f=0&a.e.\mbox{ in }Q,\\
w=0&\mbox{on }\p_pQ.
\end{array}\ri.
\]
By Theorem \ref{z5}, we have
\begin{equation}\label{eq:estmpw}
0\leq w\leq C\|f\|_{L^p(Q)},
\end{equation}
and it is easy to see that $v:=u+w$ is an $L^p$-viscosity supersolution of
\[
v_t+\P^+ (D^2v)+2^{m-1} \mu|Dv|^m=0\quad \mbox{ in }Q
\]
Thus, if we can prove that
\[
\le(\int_{J_1}v^{\e_0} \,dxdt\ri)^{\fr{1}{\e_0}}\leq C\,\under{J_2}\inf \ v,
\]
the claim will follow using \eqref{eq:estmpw}. Thus we can assume that $f=0$.

We now set $m_0:=\inf_{J_2}u$. 
We may suppose 
$m_0>0$ by adding a positive constant, which will be sent to $0$ in the end. 
Considering $v:=m_0^{-1}u$, we verify that 
$\inf_{J_2}v\leq 1$, and it is an $L^p$-viscosity supersolution of 
\[
v_t+\P^+(D^2v)+m_0^{m-1}\mu |Dv|^m=0\quad\mbox{in }Q.
\]
In view of Lemma \ref{lem:m_barrier}, 
if 
\[(2M)^{m-1}\|\mu\|_{L^q(Q)}\leq \d_2,\]
where $\d_2>0$ is from Lemma \ref{lem:m_barrier}, we can find a strong solution $\phi\in W^{2,1}_q(Q)$ of
\[
\le\{\begin{array}{rcll}
\phi_t+\P^+(D^2\phi)+(2m_0)^{m-1}\mu|D\phi|^m&\leq& g&\mbox{a.e. in }Q,\\\phi&=&0&\mbox{on }\p_pQ,\\
\phi&\geq &2&\mbox{in }K_2,\\
\mbox{supp }g&\subset &K_1.&\\
\end{array}
\ri.
\]
Then $w:=\phi-v$ is an $L^p$-viscosity subsolution of
\[
w_t+\P^-(D^2w)-(2m_0)^{m-1}\mu |Dw|^m-g=0\quad\mbox{in }Q.\]
Hence, Theorem \ref{z5} yields
\[
1\leq \under{J_2}\sup \ w\leq \under{Q}\sup \ w\leq C\,
\| g\|_{L^p(\{ (x,t) \in K_1 \ | \ (\phi-v)(x,t)\geq 0\})}, 
\]
where $C$ is a constant which depends on various absolute constants, $\d_2$, and $\| g\|_{L^p(Q)}$, which is also bounded
by various absolute constants. The above inequality now implies $|\{ (x,t)\in K_1 \ | \ v(x,t)>M\}|\leq (1-\theta) |K_1|$ 
for some $M>1$ and $\theta\in (0,1)$. 
The rest of the proof follows the arguments in the proof of Theorem 4.2 of \cite{KS4}.
\end{proof}

\mbox{ }\\
\noindent
{\bf Acknowledgements.} 
S. Koike is supported in part by Grant-in-Aid for Scientific Research (Nos.
16H06339, %\UTF{00C2}\UTF{008F}\UTF{00C2}\UTF{00C2}\UTF{0081}\UTF{00C3}\UTF{008A}\UTF{00C2}\UTF{0089}\UTF{00C2}\UTF{0092}%
16H03948, %\UTF{00C2}\UTF{0090}\UTF{00C3}\UTF{008E}\UTF{00C2}\UTF{0088}\UTF{00C3}\UTF{20AC}%
16H03946%\UTF{00C2}\UTF{0090}\UTF{00C3}\UTF{0094}\UTF{00C2}\UTF{0096}\UTF{00C3}\UTF{0098}%
) of Japan Society for Promotion Science. 
S. Tateyama is supported by Grant-in-Aid for Japan Society for Promotion Science Research Fellow 16J02399. 

%The authors thank the referees for their careful reading, and helpful comments, which help us to improve the original manuscript. 

\end{document}